\documentclass[hidelinks,onefignum,onetabnum]{siamart220329}
\pdfoutput=1

\usepackage{lipsum}
\usepackage{braket, amsfonts}
\usepackage{enumerate}
\usepackage[inline]{enumitem}
\usepackage{xcolor}
\usepackage{amssymb}
\usepackage{mathrsfs}
\usepackage{graphicx}
\usepackage{epstopdf}
\usepackage{algorithmic}
\usepackage{graphicx,epstopdf} 
\usepackage[caption=false]{subfig}
\usepackage{multirow}
\usepackage{optidef}
\ifpdf
  \DeclareGraphicsExtensions{.eps,.pdf,.png,.jpg}
\else
  \DeclareGraphicsExtensions{.eps}
\fi

\newsiamremark{remark}{Remark}
\newsiamremark{notation}{Notation}
\newsiamthm{defn}{Definition}
\newsiamremark{hypothesis}{Hypothesis}
\newsiamremark{example}{Example}
\crefname{defn}{Definition}{Definition}
\crefname{hypothesis}{Hypothesis}{Hypotheses}
\crefname{example}{Example}{Example}
\Crefname{ALC@unique}{Line}{Lines}
\newsiamthm{claim}{Claim}

\headers{Refined TSSOS}{D. Shaydurova, V. Kaibel, and S. Sager}

\title{Refined TSSOS \thanks{\funding{This work was funded by the the Deutsche Forschungs gemeinschaft (DFG, German Research Foundation) - 314838170, GRK 2297 MathCoRe.}}}

\author{Daria Shaydurova\thanks{Department of Mathematics, OvGU, Magdeburg, Germany 
		(\email{sager@ovgu.de}, \email{kaibel@ovgu.de}, \email{daria.shaydurova@ovgu.de}).}
	\and Volker Kaibel \footnotemark[1]
	\and Sebastian Sager \footnotemark[1]}

\usepackage{amsopn}

\DeclareMathOperator{\supp}{supp}
\DeclareMathOperator{\conv}{conv}
\DeclareMathOperator{\New}{New}

\begin{document}

\maketitle

\begin{abstract}
	The moment-sum of squares hierarchy by Lasserre has become an established technique for solving polynomial optimization problems. It provides a monotonically increasing series of tight bounds, but has well-known scalability limitations. For structured optimization problems, the term-sparsity SOS (TSSOS) approach scales much better due to block-diagonal matrices, obtained by completing the connected components of adjacency graphs. This block structure can be exploited by semidefinite programming solvers, for which the overall runtime then depends heavily on the size of the largest block. However, already the first step of the TSSOS hierarchy may result in large diagonal blocks. We suggest a new approach that refines TSSOS iterations using combinatorial optimization and results in block-diagonal matrices with reduced maximum block sizes. Numerical results on a benchmark library show the large potential for computational speedup for unconstrained and constrained polynomial optimization problems, while obtaining almost identical bounds in comparison to established methods.

\end{abstract}



\section{Introduction}

Consider the polynomial optimization problem (POP):
\begin{equation} \label{eq:Q1}
	\theta^* := \inf_{\mathbf{x}} \{f(\mathbf{x}) \; : \; \mathbf{x} \in \mathbf{K}\}, 
\end{equation}
where $f(\mathbf{x}) \in \mathbb{R}[\mathbf{x}]$ is a polynomial and $\mathbf{K} \subseteq \mathbb{R}^n$ is the basic semialgebraic set
\begin{equation}
	\mathbf{K} = \set{\mathbf{x} \in \mathbb{R}^n | g_j(\mathbf{x}) \geq 0, j = 1, \ldots, m },
\end{equation}
for some polynomials $g_j(\mathsf{x}) \in \mathbb{R}[\mathbf{x}]$, $j = 1,\ldots,m$. The moment-SOS (sum of squares) hierarchy by Lasserre \cite{Lasserre1} based on Putinar's certificate of positivity on $\mathbf{K}$ \cite{Putinar} is an established approach for solving this class of problems. It results in a hierarchy of  semidefinite programing (SDP) relaxations of \cref{eq:Q1}. The assosiated monotone sequence of optimal values converges to $\theta^*$ from below and the convergence is finite \emph{generically} \cite{Nie}. However, in view of the present status of SDP solvers, the moment-SOS hierarchy is limited to problems of modest size. 

There are several existing ways to address this scalability issue. One possibility is to use weaker certificates of positivity such as Krivine-Stengle's certificate \cite{Krivine,Stengle} producing a hierarchy of linear programming (LP) relaxations. Although modern LP solvers can solve huge problems with millions of variables and constraints, it has been shown that LP-relaxations provide less accurate bounds and in general have only asymptotic, not finite convergence \cite{Lasserre2}. Alternative approaches based on weaker positivity certificates are, for instance, DSOS \cite{Ahmadi}, SDSOS \cite{Ahmadi} and BSOS \cite{Lasserre3}.

Another possibility to overcome the scalability limitations is to exploit sparsity, which is often present in large-scale instances of \cref{eq:Q1}. If each polynomial in the definition of $\mathbf{K}$ involves only a few variables, and the polynomial $f$ is a sum of polynomials, each containing a few variables only, the computational cost might be reduced by considering correlative sparsity patterns \cite{CorSparsity1,CorSparsity2}, where the variables are partitioned into blocks according to the maximal cliques of the chordal extension of the graph, whose nodes correspond to the variables and there is an edge between two nodes if and only if these two variables appear in the same term of the objective polynomial $f$ or in the same polynomial $g_j$ from $\mathbf{K}$. However, if $f$ has a term involving all variables or some constraint $g_j$ contains all variables, the problem does not fulfill the correlative sparsity pattern. This means that the correlative spasity pattern fails on many fairly sparse POPs.

Instead of considering sparsity from the perspective of variables, one can exploit sparsity from the perspective of terms. The TSSOS approach from \cite{TSSOS,TSSOS2} as well as the chordal-TSSOS from \cite{ChordalTSSOS} associate a  so-called term sparsity graph with the POP. The nodes of this graph are monomials. Two nodes are connected via an edge if the product of the corresponding monomials appears in the supports of polynomials involved in the POP or is a monomial square. These methods are iterative, where each iteration consists of two successive operations: (i) a support extension operation and (ii) either a block-closure operation on adjacency matrices in the case of TSSOS or a chordal extension operation in the case of chordal-TSSOS. This two-step procedure results in a moment-SOS hierarchy with block-diagonal SDP matrices for TSSOS and quasi block-diagonal SDP matrices for chordal-TSSOS. 

Although the final iterative step of the TSSOS hierarchy is guaranteed to return the same bound as the dense moment-SOS relaxation, in practice it often happens that the same optimal value is achieved at an earlier step, even at the first one. However, for some fairly sparse polynomials, already the first iterative step of the TSSOS hierarchy gives a matrix with large diagonal blocks, so that the corresponding sparse moment-SOS relaxation might be rather expensive to solve. That is why we provide a new approach that exploits a block-diagonal matrix returned by the $k$-th iterative step of TSSOS and relies on Integer Programming (IP) to generate a new block-diagonal matrix with smaller blocks that are chosen with the goal of weakening the bound as few as possible. The maximal size of blocks is controlled using a parameter, which allows some level of flexibility. As mentioned in \cref{rem:chordal_refined} the new approach can also be used within the chordal-TSSOS method. 

The paper is organized as follows: \Cref{sec:notation} contains all necessary notation as well as the basics of the TSSOS method, summarized for the convenience of non-expert readers, experts may omit this section. In \cref{sec:reformulation}, we slightly reformulate the TSSOS procedure for generating block-diagonal matrices with the purpose of reducing the computational cost of the corresponding IP problem. In \cref{sec:refined}, we explain the idea of our approach and present the algorithm. In \cref{sec:num_res}, we test the proposed approach on randomly generated polynomials and compare its performance with the TSSOS and chordal-TSSOS methods. In \cref{sec:conclusion}, we provide conclusion and outlook on our approach. 

\section{Notation and Preliminaries}

\label{sec:notation}

Let $\mathbf{x} = (x_1, \ldots, x_n)$ be a tuple of variables and $\mathbb{R}[\mathbf{x}] = \mathbb{R}[x_1, \ldots, x_n]$ be the ring of real $n$-variate polynomials. A polynomial $f \in \mathbb{R}[\mathbf{x}]$ can be written as $f(\mathbf{x}) = \sum_{\mathbf{\alpha} \in \mathscr{A}} f_{\mathbf{\alpha}} \mathbf{x}^{\mathbf{\alpha}}$ with $f_{\mathbf{\alpha}} \in \mathbb{R}$, $\mathbf{x}^{\mathbf{\alpha}} = x_1^{\alpha_1} \cdots x_n^{\alpha_n}$ and $\mathscr{A} \subseteq \mathbb{N}^n$. The support of $f$ is defined by $\supp(f) = \set{\mathbf{\alpha} \in \mathscr{A} | f_\mathbf{\alpha} \neq 0}$ and the convex hull of $\mathscr{A}$ is denoted by $\conv(\mathscr{A})$. For a nonempty finite set $\mathscr{A} \subseteq \mathbb{N}^n$, let $\mathscr{P}(\mathscr{A})$ be the set of polynomials in $\mathbb{R}[\mathbf{x}]$ whose supports are contained in $\mathscr{A}$, i.e., $\mathscr{P}(\mathscr{A}) = \set{f \in \mathbb{R}[\mathbf{x}]  |  \supp(f) \subseteq \mathscr{A}}$. We use $|\cdot|$ to denote the cardinality of a set. For $\mathscr{A}_1, \mathscr{A}_2 \subseteq \mathbb{N}^n$, let $\mathscr{A}_1+\mathscr{A}_2 := \set{\mathbf{\alpha}_1 + \mathbf{\alpha}_2  |  \mathbf{\alpha}_1 \in \mathscr{A}_1, \mathbf{\alpha}_2 \in \mathscr{A}_ 2}$.  For any $\mathbf{\alpha} \in \mathbb{N}^n$, we define $(\alpha)_2 := ((\alpha_1 \mod 2), \ldots, (\alpha_n \mod 2)) \in \mathbb{Z}_2^n$ with $\mathbb{Z}_2$ being the ring of integers modulo 2 and refer to $(\alpha)_2$ as the parity type of $\alpha$. We also use the same notation for any subset $\mathscr{A} \subseteq \mathbb{N}^n$, i.e., $(\mathscr{A})_2 := \set{(\alpha)_2  | \alpha \in \mathscr{A}} \subseteq \mathbb{Z}_2^n$.

For $d \in \mathbb{N}$, let $\mathbb{N}_d^n := \set{\mathbf{\alpha} = (\alpha_i) \in \mathbb{N}^n  | \sum_{i=1}^{n} \alpha_i \leq d}$ and assume that $f \in \mathscr{P}(\mathbb{N}_{2d}^n)$. The \emph{sum of squares} (SOS) condition $f(\mathbf{x}) = \sum_{i=1}^{t} f_i(\mathbf{x})^2$, $f_1(\mathbf{x}), \ldots, f_t(\mathbf{x}) \in \mathbb{R}[\mathbf{x}]$ is equivalent to the existence of a positive semidefinite matrix $Q$ (called a \emph{Gram matrix} \cite{Choi}) such that 
\begin{equation}\label{eq:SOScond}
	f(\mathbf{x}) = \left(\mathbf{x}^{\mathbb{N}_d^n}\right)^T Q  \mathbf{x}^{\mathbb{N}_d^n},
\end{equation} 
where $\mathbf{x}^{\mathbb{N}_d^n}$ is the $|\mathbb{N}_d^n|$-dimensional column vector consisting of the monomials $\mathbf{x}^{\mathbf{\alpha}}$, $\mathbf{\alpha} \in \mathbb{N}_d^n$. We refer to $\mathbf{x}^{\mathbb{N}_d^n}$ as the \emph{standard monomial basis}. The computational cost of checking sum of squares conditions of multivariate polynomials can be reduced using the \emph{Newton polytope method}, where the set $\mathbb{N}_d^n$ in \cref{eq:SOScond} is replaced by
\begin{equation}
	\mathscr{B} = \frac{1}{2} \cdot \New(f) \cap \mathbb{N}^n \subseteq \mathbb{N}_d^n,
\end{equation}
with $\New(f) = \conv(\{\mathbf{\alpha} : \mathbf{\alpha} \in \supp(f)\})$ being the \emph{Newton polytope} of $f$. See Theorem 1 in \cite{Reznick}. Replacing  $\mathbb{N}_d^n$ with $\mathscr{B}$ reduces the size of the corresponding matrix $Q$ thus simplifying the semidefinite program to be solved. We further abuse notation and denote the monomial basis $\mathbf{x}^{\mathscr{B}}$ by the exponent set $\mathscr{B}$.

For a positive integer $r$, the set of $r \times r$ symmetric matrices is denoted by $\mathbb{S}^r$ and the set of $r \times r$ positive semidefinite (PSD) matrices is denoted by $\mathbb{S}^r_+$. For matrices $A,B \in \mathbb{S}^r$, let $A \circ B \in \mathbb{S}^r$ denote the Hadamard, or entrywise, product of $A$ and $B$, defined by $[A \circ B]_{ij} = A_{ij} B_{ij}$. Let $\mathbb{Z}_2^{r \times r}$ with $\mathbb{Z}_2 := \set{0,1}$ be the set of $r \times r$ binary matrices. The support of a binary matrix $B \in \mathbb{S}^r \cap \mathbb{Z}_2^{r \times r}$ is the set of locations of nonzero entries, i.e.,
\begin{equation}
	\supp(B) := \set{(i,j) \in [r] \times [r]  | B_{ij} = 1},
\end{equation}
where $[r] = \set{1, \ldots, r}$. For a symmetric binary matrix $B \in \mathbb{S}^r \cap \mathbb{Z}_2^{r \times r}$, we define the set of PSD matrices with sparsity pattern represented by $B$ as
\begin{equation}
	\mathbb{S}_+^r(B) := \set{Q \in \mathbb{S}_+^r  |  B \circ Q = Q}.
\end{equation}

For a binary matrix $B \in \mathbb{S}^r \cap \mathbb{Z}_2^{r \times r}$, let $R \subseteq [r] \times [r]$ be the adjacency relation of $B$, i.e., $(i,j) \in R$ if and only if $B_{ij} = 1$. The \emph{transitive closure} of $R$, denoted by $\overline{R}$, is the smallest relation that contains $R$ and is transitive, i.e., $(i,j),(j,k) \in \overline{R}$ implies $(i,k) \in \overline{R}$. The \emph{block-closure} $\overline{B} \in \mathbb{S}^r \cap \mathbb{Z}_2^{r \times r}$ of $B$ is defined as
	\begin{equation}\label{eq:block_closure}
	\overline{B}_{ij} := \begin{cases} 1, & (i,j) \in \overline{R}, \\
		0, & \text{otherwise}.
	\end{cases}
\end{equation}
The definition of $\overline{B}$ has a graphical description: if $G$ is the adjacency graph of $B$, then $\overline{B}$ is the adjacency matrix of the graph obtained by completing the connected components of $G$ to complete subgraphs. The matrix $\overline{B}$ is block-diagonal up to permutation and each of its blocks corresponds to a connected component of $G$.
\begin{example}
	\label{ex:block_closure}
	Let us consider the matrix
	\begin{equation*}
	B = \begin{bmatrix}
		1 & 0 & 1 & 1 & 0 \\
		0 & 1 & 0 & 1 & 0 \\ 
		1 & 0 & 1 & 0 & 0 \\ 
		1 & 1 & 0 & 1 & 0 \\ 		
		0 & 0 & 0 & 0 & 1
	\end{bmatrix}
	\end{equation*}
	The adjacency graph $G$ of $B$ has two connected components: $\set{1, 2, 3, 4}$ and $\set{5}$. Completion of these connected components to complete subgraphs results in the graph $\overline{G}$, whose adjacency matrix $\overline{B}$ has two blocks of size 4 and 1 corresponding to the connected components of $G$. The graphs $G$, $\overline{G}$ as well as the matrix $\overline{B}$ are given in \cref{fig:blockclosure}.
	
	\begin{figure}[tbhp]
		$\vcenter{\hbox{\includegraphics[trim = 150 100 130 70, clip, width=0.33\linewidth]{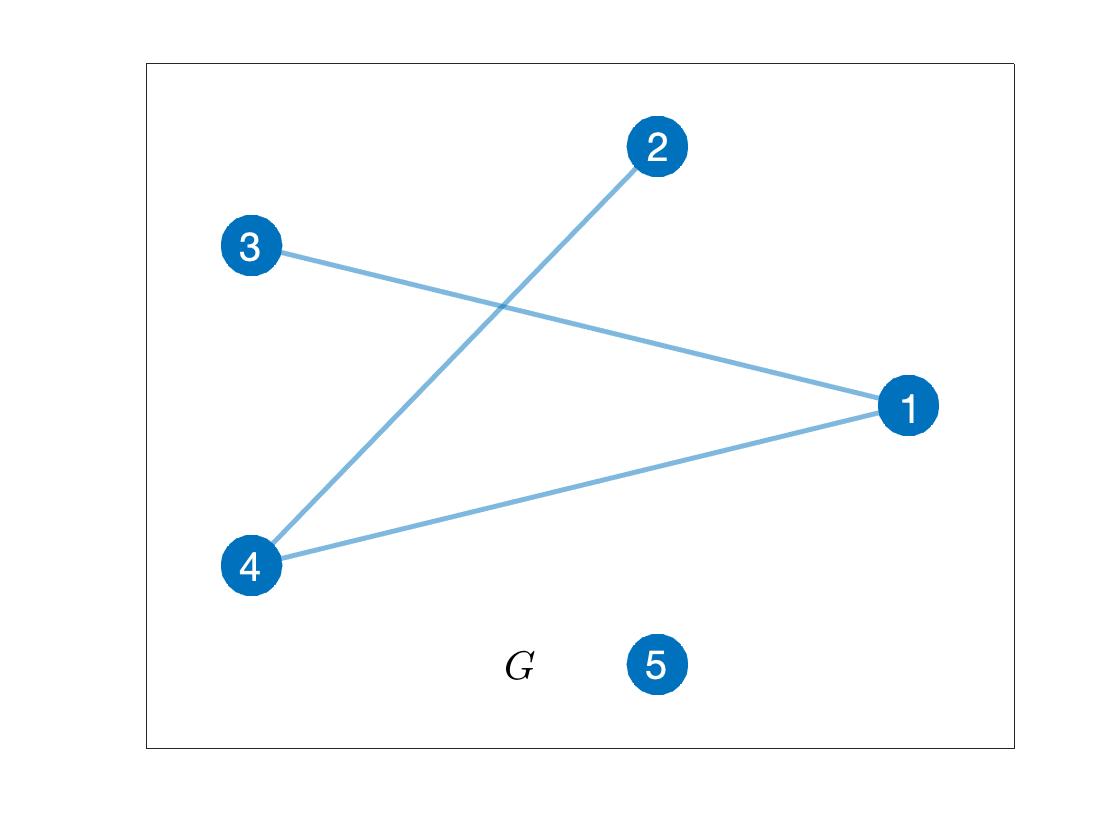}}} \quad \vcenter{\hbox{\includegraphics[trim = 150 100 130 70, clip, width=0.33\linewidth]{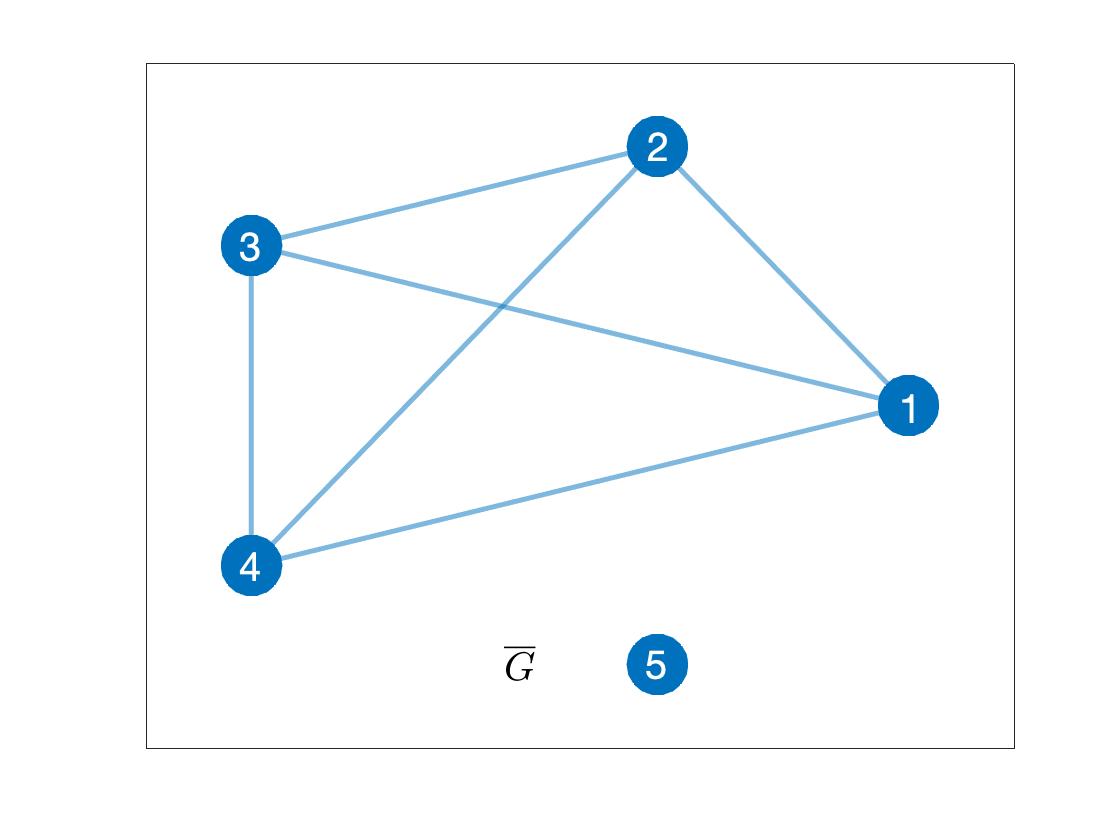}}}$ $\quad \overline{B} = \begin{bmatrix}
			1 & 1 & 1 & 1 & 0 \\
			1 & 1 & 1 & 1 & 0 \\ 
			1 & 1 & 1 & 1 & 0 \\ 
			1 & 1 & 1 & 1 & 0 \\ 		
			0 & 0 & 0 & 0 & 1
		\end{bmatrix} $
		\caption{The graphs $G$, $\overline{G}$ and the matrix $\overline{B}$ from Example \ref{ex:block_closure}}
		\label{fig:blockclosure}
	\end{figure}	
\end{example}

The rest of this section is dedicated to the basics of the TSSOS approach. Before we proceed with the description of the method, we first list some necessary definitions.

\emph{The Riesz linear functional:} Given a real sequence $\mathbf{y} = (y_{\mathbf{\alpha}})_{\mathbf{\alpha} \in \mathbb{N}^n} \subset \mathbb{R}$, let $L_{\mathbf{y}}: \mathbb{R}[\mathbf{x}] \to \mathbb{R}$ be the linear functional, defined by:
\begin{equation}
	f = \sum_{\mathbf{\alpha} \in \mathbb{N}^n} f_{\mathbf{\alpha}} \mathbf{x}^{\mathbf{\alpha}} \to L_{\mathbf{y}}(f) = \sum_{\mathbf{\alpha} \in \mathbb{N}^n} f_{\mathbf{\alpha}} y_{\mathbf{\alpha}}, \quad \forall f \in \mathbb{R}[\mathbf{x}].
\end{equation}

\emph{Moment matrix:} Given a sequence $\mathbf{y} = (y_{\mathbf{\alpha}})$ and a monomial basis $\mathscr{B}$, let $M_{\mathscr{B}}(\mathbf{y})$ be the real symmetric moment marix with rows and columns labeled by $\mathbf{\alpha} \in \mathscr{B}$ and constructed as follows:
\begin{equation}
	M_{\mathscr{B}}(\mathbf{y})_{\alpha \beta} := L_{\mathbf{y}}(\mathbf{x}^{\mathbf{\alpha}}\mathbf{x}^{\mathbf{\beta}}) = y_{\alpha+\beta}, \quad \forall \alpha, \beta \in \mathscr{B}.
\end{equation}
If $\mathscr{B} = \mathbb{N}_d^n$, we also denote $M_{\mathscr{B}}(\mathbf{y})$ by $M_d(\mathbf{y})$. 

\emph{Localizing matrix:} Given a polynomial $g = \sum_{\mathbf{\gamma}} g_{\mathbf{\gamma}} \mathbf{x}^{\mathbf{\gamma}} \in \mathbb{R}[\mathbf{x}]$ and an integer $d \geq 1$, the localizing matrix with respect to $\mathbf{y} = (y_{\mathbf{\alpha}})$ and $g$ is defined to be the matrix $M_d(g \mathbf{y})$ with rows and columns indexed by $\mathbf{\alpha} \in \mathbb{N}_d^n$ such that
\begin{equation}
	M_d(g \mathbf{y})_{\mathbf{\alpha} \mathbf{\beta}} := L_{\mathbf{y}}(g(\mathbf{x})\mathbf{x}^{\mathbf{\alpha}} \mathbf{x}^{\mathbf{\beta}}) = \sum_{\mathbf{\gamma}} g_{\mathbf{\gamma}} y_{\mathbf{\gamma}+\mathbf{\alpha}+\mathbf{\beta}}, \quad \forall \mathbf{\alpha}, \mathbf{\beta} \in \mathbb{N}_d^n.
\end{equation}

\subsection{The TSSOS method \cite{TSSOS}}

\label{subsec:tssos}

Consider a polynomial $f(\mathbf{x}) \in \mathbb{R}[\mathbf{x}]$ with $\mathscr{A} = \supp(f)$ and let $\mathscr{B}$ be a monomial basis with $r = |\mathscr{B}|$. Let $\mathscr{S}^{(0)} = \mathscr{A} \cup (2 \mathscr{B})$, where $2 \mathscr{B} = \set{2\mathbf{\beta} |  \mathbf{\beta} \in \mathscr{B}}$. For $k \geq 1$, the TSSOS method recursively defines binary matrices $B_{\mathscr{A}}^{(k)} \in \mathbb{S}^r \cap \mathbb{Z}_2^{r \times r}$ indexed by $\mathscr{B}$ via two successive steps:
\begin{enumerate}[label={\arabic*)}]
	\item \textbf{Support-extension}: define a binary matrix $C_{\mathscr{A}}^{(k)} \in \mathbb{S}^r \cap \mathbb{Z}_2^{r \times r}$ by
	\begin{equation} \label{eq:CA}
		[C_{\mathscr{A}}^{(k)}]_{\mathbf{\beta} \mathbf{\gamma}} := \begin{cases}
			1, & \text{if} \; \mathbf{\beta} + \mathbf{\gamma} \in \mathscr{S}^{(k-1)}, \\
			0, & \text{otherwise}.
		\end{cases}
	\end{equation}
	\item \textbf{Block-closure}: construct $\overline{C_{\mathscr{A}}^{(k)}}$ using \cref{eq:block_closure}, set $B_{\mathscr{A}}^{(k)} = \overline{C_{\mathscr{A}}^{(k)}}$ and $\mathscr{S}^{(k)} = \set{\mathbf{\beta} + \mathbf{\gamma}  |  [B_{\mathscr{A}}^{(k)}]_{\mathbf{\beta} \mathbf{\gamma}} = 1}$.
\end{enumerate}

For all $k \geq 1$, $\supp(B_{\mathscr{A}}^{(k)}) \subseteq \supp(B_{\mathscr{A}}^{(k+1)})$. Hence the sequence of binary matrices $\left(B_{\mathscr{A}}^{(k)}\right)_{k \geq 1}$ stabilizes after a finite number of steps. The stabilized matrix is denoted by $B_{\mathscr{A}}^{(*)}$.	

Let $\Sigma(\mathscr{A})$ be the set of SOS polynomials supported on $\mathscr{A}$, i.e.,
\begin{equation}\label{eq:Sigma_A}
	\Sigma(\mathscr{A}) := \set{f \in \mathscr{P}(\mathscr{A}) |  \exists Q \in \mathbb{S}_+^r \; \text{s.t.} \; f = \left(\mathbf{x}^{\mathscr{B}}\right)^T Q \mathbf{x}^{\mathscr{B}} }.
\end{equation}
For $k \geq 1$, let $\Sigma_k(\mathscr{A})$ be the subset of $\Sigma(\mathscr{A})$ whose member admits a Gram matrix with sparsity pattern represented by $B_{\mathscr{A}}^{(k)}$, i.e.,
\begin{equation} 
	\Sigma_k(\mathscr{A}) := \set{f \in \mathscr{P}(\mathscr{A})  |  \exists Q \in \mathbb{S}_+^r \left(B_{\mathscr{A}}^{(k)}\right)  \; \text{s.t.} \; f = \left(\mathbf{x}^{\mathscr{B}}\right)^T Q \mathbf{x}^{\mathscr{B}}}.
\end{equation}
In addition, let 
\begin{equation}
	\Sigma_*(\mathscr{A}) := \set{f \in \mathscr{P}(\mathscr{A})  |  \exists Q \in \mathbb{S}_+^r \left(B_{\mathscr{A}}^{(*)}\right)  \; \text{s.t.} \; f = \left(\mathbf{x}^{\mathscr{B}}\right)^T Q \mathbf{x}^{\mathscr{B}} }.
\end{equation}
It was shown in \cite{TSSOS} that for a finite set $\mathscr{A} \subseteq \mathbb{N}^n$, one has $\Sigma_*(\mathscr{A}) = \Sigma(\mathscr{A})$. So the TSSOS method gives the following chain of inclusions:
\begin{equation}\label{eq:inc}
	\Sigma_1(\mathscr{A}) \subseteq \Sigma_2(\mathscr{A}) \subseteq \cdots \subseteq \Sigma_*(\mathscr{A}) = \Sigma(\mathscr{A}).
\end{equation}

\begin{remark}
	 The block-closure operation from \cref{eq:block_closure} used in the second step of the TSSOS algorithm to obtain $B_{\mathscr{A}}^{(k)}$ can be replaced with a chordal-extension operation on the adjacency graph of $C_{\mathscr{A}}^{(k)}$. This results in the method called chordal-TSSOS. See \cite{ChordalTSSOS} for more details on this approach.
\end{remark}

\subsubsection{A block SDP hierarchy for the unconstrained case}

Consider the unconstrained polynomial optimization problem:
\begin{equation} \label{eq:P}
	\theta^* := \inf_{\mathbf{x}} \{f(\mathbf{x}) \; : \; \mathbf{x} \in \mathbb{R}^n\} \tag{P}
\end{equation}
with $f(\mathbf{x}) \in \mathbb{R}[\mathbf{x}]$. The SOS relaxation of \eqref{eq:P} is given by
\begin{equation} \label{eq:SOS}
	\theta_{sos} := \sup_{\lambda} \{\lambda \; | \; f(\mathbf{x}) - \lambda \in \Sigma(\mathscr{A})\}, \tag{SOS}
\end{equation}
with $\mathscr{A} = \{0\} \cup \supp(f)$. Raplacing $\Sigma(\mathscr{A})$ with $\Sigma_k(\mathscr{A})$ yields a hierarchy of sparse SOS relaxations of \eqref{eq:P}:
\begin{equation} \label{eq:Pk}
	(P^k)^* : \quad \theta_k := \sup_{\lambda} \{\lambda \; | \; f(\mathbf{x}) - \lambda \in \Sigma_k(\mathscr{A})\}, \quad k = 1,2,\ldots.
\end{equation}
In addition, let 
\begin{equation} \label{eq:TSSOS}
	\theta_{tssos} := \sup_{\lambda} \{\lambda \; | \; f(\mathbf{x}) - \lambda \in \Sigma_*(\mathscr{A})\}. \tag{TSSOS}
\end{equation}
It follows  from \eqref{eq:inc} that the hierarchy of sparse SOS relaxations \eqref{eq:Pk} results in the hierarchy of lower bounds for the optimum of \eqref{eq:P}:
\begin{equation}
	\theta^* \geq \theta_{sos} = \theta_{tssos} \geq \cdots \geq \theta_2 \geq \theta_1.
\end{equation}
Let $\mathscr{B}$ be the monomial basis. For each $k \geq 1$, the dual of \eqref{eq:Pk} is the following block moment problem
\begin{equation} (P^k) : \quad 
	\begin{cases} \label{eq:Pk'}
		\inf & L_{\mathbf{y}}(f)\\
		\text{s.t.} & B_{\mathscr{A}}^{(k)} \circ M_{\mathscr{B}}(\mathbf{y}) \succeq 0,\\
		& y_{\mathbf{0}} = 1. 
	\end{cases}
\end{equation}
It was proven in \cite{TSSOS} that, for each $k \geq 1$, there is no duality gap between $(P^k)^*$ and $(P^k)$.

\subsubsection{A block SDP hierarchy for the constrained case}

\label{sec:tssos_con}

Consider the constrained polynomial optimization problem:
\begin{equation} \label{eq:Q}
	\theta^* := \inf_{\mathbf{x}} \{f(\mathbf{x}) \; : \; \mathbf{x} \in \mathbf{K}\}, \tag{Q}
\end{equation}
where $f(\mathbf{x}) \in \mathbb{R}[\mathbf{x}]$ is a polynomial and $\mathbf{K} \subseteq \mathbb{R}^n$ is the basic semialgebraic set
\begin{equation}
	\mathbf{K} = \set{\mathbf{x} \in \mathbb{R}^n | g_j(\mathbf{x}) \geq 0, j = 1, \ldots, m },
\end{equation}
for some polynomials $g_j(\mathsf{x}) \in \mathbb{R}[\mathbf{x}]$, $j = 1,\ldots,m$. 

Let $d_j = \lceil \deg(g_j)/2 \rceil$, $j = 0, \ldots, m$, where $g_0 := 1$ and $$d = \max \{\lceil \deg(f) / 2\rceil, d_1, \ldots, d_m\}.$$ With $\hat{d} \geq d$ being a positive integer, the Lasserre hierarchy \cite{Lasserre1} of moment semidefinite relaxations of \ref{eq:Q} is defined by:
\begin{equation} \label{eq:las}
	\begin{cases}
		\inf & L_{\mathbf{y}}(f)\\
		\text{s.t.} & M_{\hat{d}}(\mathbf{y}) \succeq 0,\\
		& M_{\hat{d}-d_j}(g_j\mathbf{y}) \succeq 0, \; j = 1,\ldots,m, \\
		& y_{\mathbf{0}} = 1,
	\end{cases}
\end{equation}
with optimal value denoted by $\theta_{\hat{d}}$ and $\hat{d}$ called the \emph{relaxation order}.

Let $\mathscr{A} = \supp(f) \cup \bigcup_{j=1}^{m} \supp(g_j)$. Set $\mathscr{S}_{0,\hat{d}}^{(0)} = \mathscr{A} \cup (2\mathbb{N}^n_{\hat{d}})$ , $\mathscr{S}_{j,\hat{d}}^{(0)} = \emptyset$, $j = 1, \ldots, m$ and $r_j := \binom{n+\hat{d}-d_j}{\hat{d}-d_j}$. For $k \geq 1$, the TSSOS method recursively defines binary matrices $B_{j,\hat{d}}^{(k)} \in \mathbb{S}^{r_j} \cap \mathbb{Z}_2^{r_j \times r_j}$, indexed by $\mathbb{N}_{\hat{d}-d_j}^n$, $j = 0, \ldots, m$  via two successive steps:
\begin{enumerate}[label={\arabic*)}]
	\item \textbf{Support-extension}: define a binary matrix $C_{j,\hat{d}}^{(k)} \in \mathbb{S}^{r_j} \cap \mathbb{Z}_2^{r_j \times r_j}$ with rows and columns indexed by $\mathbb{N}_{\hat{d}-d_j}^n$ as
	\begin{equation*}
		[C_{j,\hat{d}}^{(k)}]_{\mathbf{\beta}\mathbf{\gamma}} := \begin{cases}
			1, & \text{if } (\supp(g_j) + \mathbf{\beta} + \mathbf{\gamma}) \cap \bigcup_{j=0}^{m} \mathscr{S}_{j,\hat{d}}^{(k-1)} \neq \emptyset, \\
			0, & \text{otherwise}.
		\end{cases}
	\end{equation*}
	\item \textbf{Block-closure}: construct $\overline{C_{j,\hat{d}}^{(k)}}$ using \cref{eq:block_closure}, set $B_{j,\hat{d}}^{(k)} = \overline{C_{j,\hat{d}}^{(k)}}$ and
	\begin{equation*}
		\mathscr{S}^{(k)} = \supp(g_j) + \set{\mathbf{\beta} + \mathbf{\gamma}  |  [B_{j,\hat{d}}^{(k)}]_{\mathbf{\beta} \mathbf{\gamma}} = 1 }.
	\end{equation*}
\end{enumerate} 
Therefore with $k \geq 1$, the TSSOS method gives a block moment relaxations of \eqref{eq:las}:
\begin{equation}\label{eq:tssosch}
	\begin{cases}
		\inf & L_{\mathbf{y}}(f)\\
		\text{s.t.} & B_{0,\hat{d}}^{(k)} \circ M_{\hat{d}}(\mathbf{y}) \succeq 0,\\
		& B_{j,\hat{d}}^{(k)} \circ  M_{\hat{d}-d_j}(g_j\mathbf{y}) \succeq 0, \; j = 1,\ldots,m, \\
		& y_{\mathbf{0}} = 1,
	\end{cases}
\end{equation}
with optimal value denoted by $\theta_{\hat{d}}^{(k)}$. By construction, for all $k \geq 1$ and $j = 0, \ldots, m$,  $\supp(B_{j,\hat{d}}^{(k)}) \subseteq \supp(B_{j,\hat{d}}^{(k+1)})$. Hence the sequence of binary matrices $\left(B_{j,\hat{d}}^{(k)}\right)_{k \geq 1}$ stabilizes for all $j$ after a finite number of steps. The stabilized matrices are denoted by $B_{j,\hat{d}}^{(*)}$, $j = 0, \ldots, m$  and the optimal value of the corresponding SDP is denoted by $\theta_{\hat{d}}^{*}$. 

It was shown in \cite{TSSOS} that for fixed $\hat{d} \geq d$, the sequence $(\theta_{\hat{d}}^{(k)})_{k \geq 1}$ of optimal values of \eqref{eq:tssosch} is monotone nondecreasing and $\theta_{\hat{d}}^{*} = \theta_{\hat{d}}$.

\section{Reformulation of the TSSOS}

\label{sec:reformulation}

In this section we present an alternative procedure for generating block-diagonal binary matrices $B_{\mathscr{A}}^{(k)}$ and $B_{j,\hat{d}}^{(k)}$, $j = 0, \ldots, m$ from the TSSOS method. Within this procedure instead of working with binary matrices $C_{\mathscr{A}}^{(k)}$ and $C_{j,\hat{d}}^{(k)}$, $j = 0, \ldots, m$  indexed by $\mathscr{B}$ and $\mathscr{B}^{(j)}$, respectively, we apply the block-closure operation to binary matrices $C^{(k)}$ and $C^{(j,k)}$, $j = 0, \ldots, m$  indexed by $\left(\mathscr{B}\right)_2$ and $\left(\mathscr{B}^{(j)}\right)_2$. Matrices $C^{(k)}$ and $C^{(j,k)}$ are further used in our approach. 

\subsection{Defining matrices $B_{\mathscr{A}}^{(k)}$ for the unconstrained case}

Let $f(\mathbf{x}) \in \mathbb{R}[\mathbf{x}]$ with $\mathscr{A} = \supp(f)$ and let $\mathscr{B}$ be a monomial basis with $r = |\mathscr{B}|$ and set $\mathscr{S}^{(0)} = \mathscr{A} \cup (2 \mathscr{B})$.  We partition the set $\mathscr{B}$ into subsets $\mathscr{B}_{\delta}$ with $\delta \in \left(\mathscr{B}\right)_2$ defined in the following way:
\begin{equation}
	\mathscr{B}_{\delta} = \set{\mathbf{\alpha} \in \mathscr{B} | (\mathbf{\alpha})_2 = \delta}.
\end{equation}
Using this notation we have $\mathscr{B} = \bigcup_{\delta \in \left(\mathscr{B}\right)_2} \mathscr{B}_{\delta}$. For $k \geq 1$, we define binary matrices $B_{\mathscr{A}}^{(k)} \in \mathbb{S}^r \cap \mathbb{Z}_2^{r \times r}$ indexed by $\mathscr{B}$ via three successive steps:
\begin{enumerate}[label={\arabic*)}]
	\item \textbf{Support-extension}: define a binary matrix $C^{(k)} \in \mathbb{S}^{r_b} \cap \mathbb{Z}_2^{{r_b}  \times {r_b} }$, $r_b = |\left(\mathscr{B}\right)_2|$, indexed by $\left(\mathscr{B}\right)_2$ as
	\begin{equation} \label{eq:C}
		[C^{(k)}]_{\mathbf{\delta}\mathbf{\sigma}} := \begin{cases}
			1, & \text{if }  (\mathscr{B}_{\delta} +  \mathscr{B}_{\sigma}) \cap \mathscr{S}^{(k-1)} \neq \emptyset, \\
			0, & \text{otherwise.}
		\end{cases}
	\end{equation}
	\item \textbf{Block-closure}: evaluate $\overline{C^{(k)}}$ using \cref{eq:block_closure}, set $B^{(k)} = \overline{C^{(k)}}$ and $\mathscr{S}^{(k)} = \set{\mathscr{B}_{\delta} + \mathscr{B}_{\sigma} | [B^{(k)}]_{\delta \sigma} = 1}$
	\item \textbf{Reconstruction of $B_{\mathscr{A}}^{(k)} \in \mathbb{S}^r \cap \mathbb{Z}_2^{r \times r}$ from $B^{(k)}$}:
	\begin{equation} \label{eq:B}
		[B_{\mathscr{A}}^{(k)}]_{\beta \gamma} = 
		\begin{cases}
			1, & \text{if }  [B^{(k)}]_{\delta \sigma} = 1 \; \text{with} \; \delta = (\beta)_2, \sigma = (\gamma)_2, \\
			0, & \text{otherwise.}
		\end{cases}
	\end{equation}
\end{enumerate}
\begin{lemma}\label{lm:refuncon}
	The three-step-procedure presented above generates the same matrices $B_{\mathscr{A}}^{(k)}$, $k \geq 1$ as the two-step-procedure  described in \cref{subsec:tssos}.
\end{lemma}
\begin{proof}
	To prove the equivalence between the two- and three-step-procedures, we need to show that if $[B_{\mathscr{A}}^{(k)}]_{\beta \gamma} = 1$ for some $\beta, \gamma \in \mathscr{B}$ with $\delta := (\beta)_2$, $\sigma := (\gamma)_2$, then $[B_{\mathscr{A}}^{(k)}]_{\omega \nu} = 1$ for all $\omega \in \mathscr{B}_{\delta}$, $\nu \in \mathscr{B}_{\sigma}$, which will imply that we can first define a binary matrix $C^{(k)}$ indexed by $\delta, \sigma \in (\mathscr{B})_2$, evaluate its block-closure $B^{(k)}$, and reconstruct $B_{\mathscr{A}}^{(k)}$ using \cref{eq:B}.
	
	For $k \geq 1$, let $R$ be the adjacency relation of $C_{\mathscr{A}}^{(k)}$. Since $B_{\mathscr{A}}^{(k)} = \overline{C_{\mathscr{A}}^{(k)}}$, $\overline{R}$ is the adjacency relation of $B_{\mathscr{A}}^{(k)}$ and $[B_{\mathscr{A}}^{(k)}]_{\beta \gamma} = 1$ implies $(\beta, \gamma) \in \overline{R}$. For all $\omega \in \mathscr{B}_{\delta}$, we have $(\omega + \beta)_2 = (\omega)_2 + (\beta)_2 = 2 \delta= \mathbf{0}$, which means $\omega + \beta \in 2 \mathscr{B}$ and, consequently, $(\omega,\beta) \in R \subseteq \overline{R}$.  Similarly, for all $\nu \in \mathscr{B}_{\sigma}$, we have $(\gamma + \nu)_2 = \mathbf{0}$ and $(\gamma,\nu) \in R \subseteq \overline{R}$. Since for all $\omega \in \mathscr{B}_{\delta}$, $(\omega,\beta), (\beta, \gamma) \in \overline{R}$, it follows from transitivity of $\overline{R}$ that $(\omega, \gamma) \in \overline{R}$ for all $\omega \in \mathscr{B}_{\delta}$, which in combination with $(\gamma,\nu) \in \overline{R}$, $\nu \in \mathscr{B}_{\sigma}$ gives $(\omega, \nu) \in \overline{R}$ for all $\omega \in \mathscr{B}_{\delta}$, $\nu \in \mathscr{B}_{\sigma}$. Applying \cref{eq:block_closure} we get $[B_{\mathscr{A}}^{(k)}]_{\omega \nu} = 1$ for all $\omega \in \mathscr{B}_{\delta}$, $\nu \in \mathscr{B}_{\sigma}$.
\end{proof}

\begin{remark}
	 Apart from the case, when all sets $\mathscr{B}_{\delta}$, $\delta \in (\mathscr{B})_2$ have cardinality one, implying $|\mathscr{B}| = \sum_{\delta \in (\mathscr{B})_2} |\mathscr{B}_{\delta}| = |(\mathscr{B})_2|$, a binary matrix $C^{(k)}$ defined in \cref{eq:C} is smaller than $C_{\mathscr{A}}^{(k)}$ defined in \cref{eq:CA}, for example, for a polynomial with the standard monomial basis $\mathscr{B} = \mathbb{N}_4^8$, we have $C_{\mathscr{A}}^{(k)} \in \mathbb{S}^{495}$ and $C^{(k)} \in \mathbb{S}^{163}$.
\end{remark}

\subsection{Defining matrices $B_{j, \hat{d}}^{(k)}$ for the constrained case}
Let $\mathscr{A} = \supp(f) \cup \bigcup_{j=1}^{m} \supp(g_j)$. For the relaxation order $\hat{d}$, set $\mathscr{S}_{0,\hat{d}}^{(0)} = \mathscr{A} \cup (2\mathbb{N}^n_{\hat{d}})$ , $\mathscr{S}_{j,\hat{d}}^{(0)} = \emptyset$, $j = 1, \ldots, m$ . Let $d_j = \lceil \deg(g_j)/2 \rceil$, $j = 0, \ldots, m$ , where $g_0 := 1$ and $r_j := \binom{n+\hat{d}-d_j}{\hat{d}-d_j}$. We partition each $\mathscr{B}^{(j)} = \mathbb{N}_{\hat{d}-d_j}^n$, $j = 0, \ldots, m$ into subsets $\mathscr{B}_{\delta}^{(j)} = \set{\alpha \in \mathscr{B}^{(j)} | (\alpha)_2 = \delta}$ with $\delta \in \left(\mathscr{B}^{(j)}\right)_2$ and define $r_{bj} := |\left(\mathscr{B}^{(j)}\right)_2|$. For $k \geq 1$, we recursively define binary matrices $B_{j,\hat{d}}^{(k)} \in \mathbb{S}^{r_j} \cap \mathbb{Z}_2^{r_j \times r_j}$, indexed by $\mathscr{B}^{(j)}$, $j = 0, \ldots, m$  via three successive steps:
\begin{enumerate}[label={\arabic*)}]
	\item \textbf{Support-extension}: define a binary matrix $C^{(j,k)} \in \mathbb{S}^{r_{bj}} \cap \mathbb{Z}_2^{r_{bj}  \times r_{bj} }$ indexed by $\left(\mathscr{B}^{(j)}\right)_2$ as
	\begin{equation} \label{eq:CC}
		[C^{(j,k)}]_{\mathbf{\delta}\mathbf{\sigma}} := \begin{cases}
			1, & \text{if }  (\supp(g_j) + \mathscr{B}^{(j)}_{\delta} +  \mathscr{B}^{(j)}_{\sigma}) \cap   \bigcup_{j=0}^{m}\mathscr{S}_{j,\hat{d}}^{(k-1)} \neq \emptyset, \\
			0, & \text{otherwise.}
		\end{cases}
	\end{equation}
	\item \textbf{Block-closure}: evaluate $\overline{C^{(j,k)}}$ using \cref{eq:block_closure}, set $B^{(j,k)} = \overline{C^{(j,k)}}$ and $$\mathscr{S}_{j,\hat{d}}^{(k)} = \supp(g_j) + \set{\mathscr{B}^{(j)}_{\delta} + \mathscr{B}^{(j)}_{\sigma} | [B^{(j,k)}]_{\delta \sigma} = 1}.$$
	\item \textbf{Reconstruction of $B_{j,\hat{d}}^{(k)} \in \mathbb{S}^{r_j} \cap \mathbb{Z}_2^{r_j \times r_j}$ from $B^{(j,k)}$}:
	\begin{equation*}
		[B_{j,\hat{d}}^{(k)}]_{\beta \gamma} = 
		\begin{cases}
			1, & \text{if }  [B^{(j,k)}]_{\delta \sigma} = 1 \; \text{with} \; \delta = (\beta)_2, \sigma = (\gamma)_2, \\
			0, & \text{otherwise.}
		\end{cases}
	\end{equation*}
\end{enumerate}

\begin{lemma}
	For the relaxation order $\hat{d}$, the three-step-procedure presented above generates the same matrices $B_{j,\hat{d}}^{(k)}$, $k \geq 1$, $j = 0, \ldots, m$ as the two-step-procedure  described in \cref{subsec:tssos}.
\end{lemma}
\begin{proof}
Follows from applying the arguments from the proof of \cref{lm:refuncon} to each matrix $C_{j,\hat{d}}^{(k)}$, $k \geq 1$, $j = 0, \ldots, m$ .
\end{proof}

\subsection{Examples}

\begin{example}\label{ex:ex1}
	Consider the polynomial $f(\mathbf{x}) = x_1^6 + 3 x_2^6 + 5 x_3^6 - 3 x_1^5 + 7 x_1^3 x_3^2 + 8 x_1 x_3^4 - 6 x_1 x_3^2 + 5$. The monomial basis is $\mathscr{B} = \mathbb{N}_3^3$ with $$(\mathscr{B})_2 = \set{(0,0,0), (1,0,0), (0,1,0), (0,0,1), (1,1,0), (1,0,1), (0,1,1), (1,1,1)}.$$  For this monomial basis, the matrix $C_{\mathscr{A}}^{(1)}$ defined in \cref{eq:CA} has size $|\mathscr{B}| = \binom{3+3}{3} = 20$ and the matrix $C^{(1)}$ defined in \cref{eq:C} has size $|(\mathscr{B})_2| = 8$. The adjacency graphs $G_{\mathscr{A}}^{(1)}$ and $G^{(1)}$ of these matrices are depicted in \cref{fig:a1,fig:b1}, respectively. Since $G^{(1)}$ has six connected components, the matrix $B^{(1)}$ has six diagonal blocks: $\set{\{000,100\},\{101,001\},\{010\},\{111\},\{110\},\{011\}}$. Applying \cref{eq:B} to $B^{(1)}$ we get the matrix $B_{\mathscr{A}}^{(1)}$ with the following six blocks:
	\begin{enumerate}[label={\arabic*)}]
		\item $\set{1, x_1, x_1^2, x_2^2, x_3^2, x_1^3, x_1 x_2^2, x_1 x_3^2}$,
		\item $\set{x_3, x_3^3, x_1^2 x_3, x_2^2x_3, x_1 x_3}$,
		\item $\set{x_2, x_2^3, x_1^2 x_2, x_2 x_3^2}$,
		\item $\set{x_1 x_2 x_3}$,
		\item $\set{x_1 x_2}$,
		\item $\set{x_2 x_3}$,
	\end{enumerate}
corresponding to the connected components of the graph $G_{\mathscr{A}}^{(1)}$.
	\begin{figure}[tbhp]
		\centering
		\subfloat[$G_{\mathscr{A}}^{(1)}$]{\label{fig:a1}\includegraphics[trim = 150 100 130 70, clip, width=0.5\linewidth]{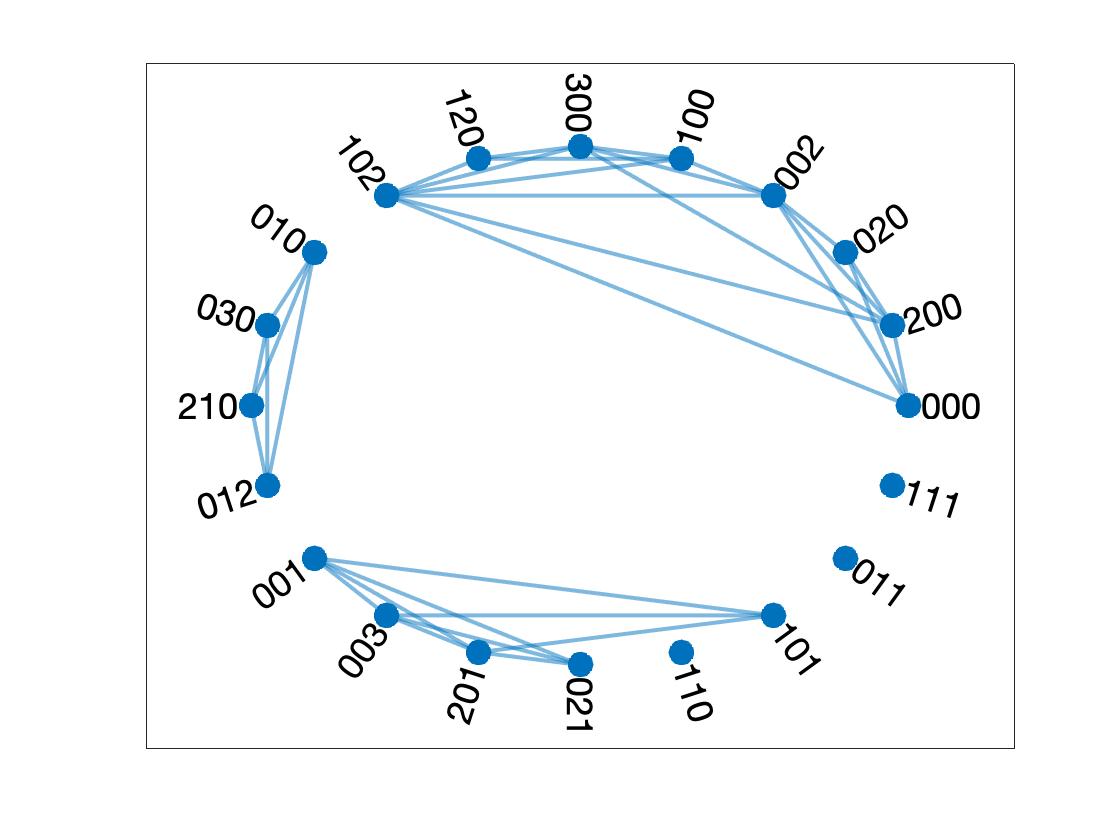}}
		\subfloat[$G^{(1)}$]{\label{fig:b1}\includegraphics[trim = 150 100 130 70, clip, width=0.5\linewidth]{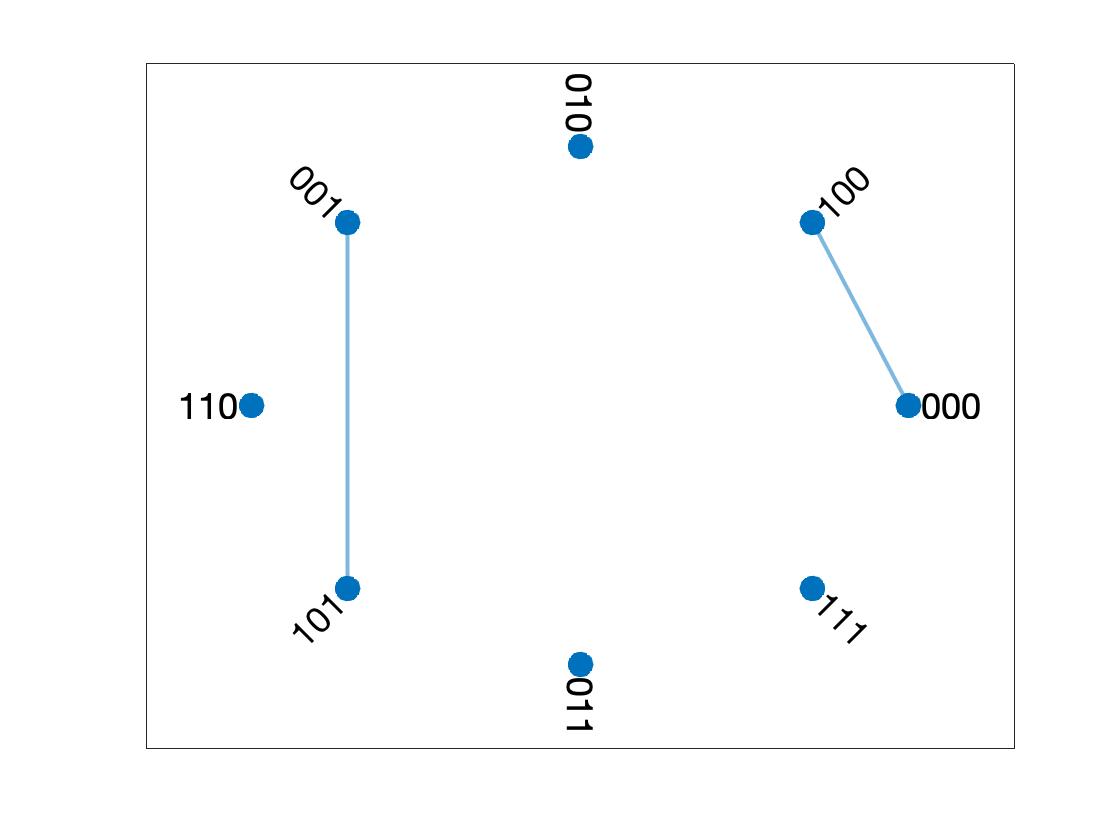}}
		\caption{Adjacency graphs of $C_{\mathscr{A}}^{(1)}$ and $C^{(1)}$ from \cref{ex:ex1}.}
		\label{fig:examp1}
	\end{figure}
	The sequence of binary matrices $\left(B_{\mathscr{A}}^{(k)}\right)_{k \geq 1}$ stabilizes at $k = 2$. Solving the corresponding SDP problems we obtain $\theta_1 = \theta_2 = \theta_{tssos} = \theta_{sos} \approx -43.8281$.
\end{example}
\begin{example}\label{ex:ex2}
	We now modify three terms of the polymonial from \cref{ex:ex1} and consider the polynomial $f(\mathbf{x}) =  x_1^6 + 3 x_2^6 + 5 x_3^6 - 3 x_1 x_2^2 x_3^2 + 7 x_2^3 x_3^2 + 8 x_1 x_3^3 - 6 x_1 x_3^2 + 5$. The monomial basis is again $\mathscr{B} = \mathbb{N}_3^3$. The adjacency graphs $G_{\mathscr{A}}^{(1)}$ and $G^{(1)}$ of matrices $C_{\mathscr{A}}^{(1)}$ and $C^{(1)}$ are depicted in \cref{fig:a2,fig:b2}, respectively. Since $G^{(1)}$ has only one connected component, $B^{(1)}$ and, consequently, $B_{\mathscr{A}}^{(1)}$ obtained using \cref{eq:B} are all-ones matrices, which coinsides with the fact that the graph $G_{\mathscr{A}}^{(1)}$ has only one connected component.
	 Solving the corresponding SDP problem we obtain $\theta_1 = \theta_{tssos} = \theta_{sos} \approx -29.6934$.
	\begin{figure}[tbhp]
		\centering
		\subfloat[$G_{\mathscr{A}}^{(1)}$]{\label{fig:a2}\includegraphics[trim = 150 100 130 70, clip, width=0.5\linewidth]{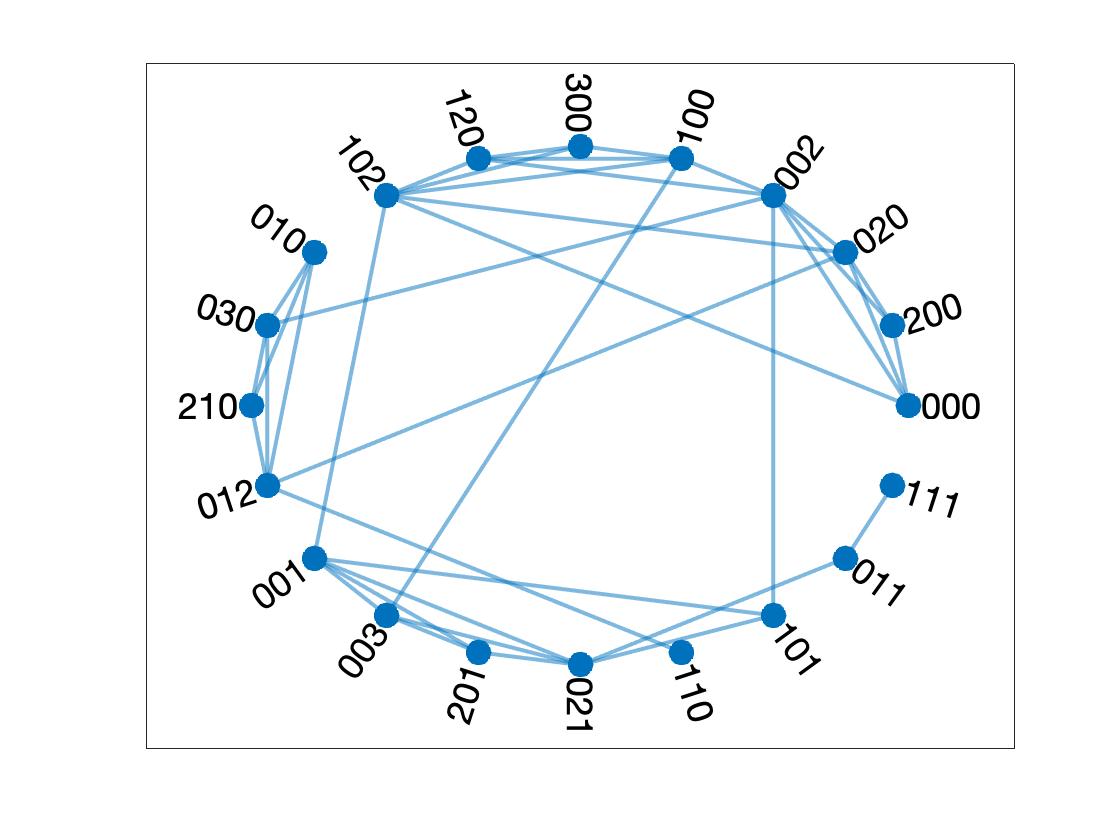}}
		\subfloat[$G^{(1)}$]{\label{fig:b2}\includegraphics[trim = 150 100 130 70, clip, width=0.5\linewidth]{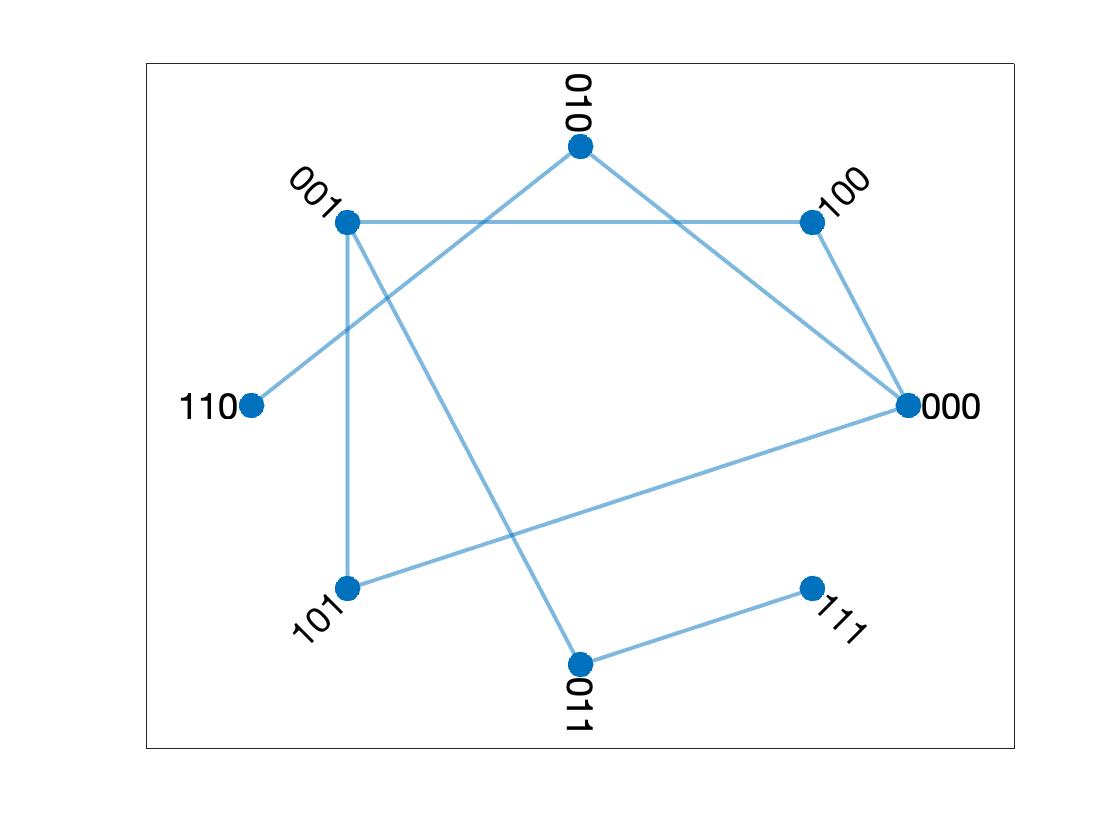}}
		\caption{Adjacency graphs of $C_{\mathscr{A}}^{(1)}$ and $C^{(1)}$ from \cref{ex:ex2}.}
		\label{fig:examp2}
	\end{figure}
\end{example}

\section{Refined TSSOS}

\label{sec:refined}

In this section, we are going to describe a new approach that exploits block-diagonal matrices returned by the TSSOS method and produces new block-diagonal matrices with reduced maximum block sizes using combinatorial optimization.  We first introduce some necessary terminology. If $P = \set{P_i : i \in [s]}$ is a partition of a set $S$, we call $\max \set{|P_i| : i \in [s]}$ the \emph{width} of the partition $P$. If a partition $P$ of a set $S$ is a refinement of a partition $P'$ of $S$, we write $P \leq P'$. 

For $k \geq 1$, let $B_{\mathscr{A}}^{(k)}$ be a binary matrix from \cref{subsec:tssos} with rows and columns indexed by elements of a monomial basis $\mathscr{B}$. Let $I^{(k)}$ be a partition of $\mathscr{B}$ induced by $B_{\mathscr{A}}^{(k)}$: two vectors $\beta, \gamma \in \mathscr{B}$ belong to the same element of $I^{(k)}$ if and only if the rows and columns indexed by $\beta$, $ \gamma$ belong to the same diagonal block of the matrix $B_{\mathscr{A}}^{(k)}$. Our approach generates a refinement $I^{(\tau)}$ of the partition $I^{(k)}$, where $\tau = k - 1 + \varepsilon$ with $\varepsilon \in (0,1)$ being a parameter that is used to control the width of the partition $I^{(\tau)}$. 

Let $B_{\mathscr{A}}^{(\tau)}$ be the block-diagonal matrix corresponding to the partition $I^{(\tau)}$, then $\supp(B_{\mathscr{A}}^{(\tau)}) \subseteq \supp(B_{\mathscr{A}}^{(k)})$.  Let $\Sigma_{\tau}(\mathscr{A})$ be the subset of $\Sigma(\mathscr{A})$ (defined in \cref{eq:Sigma_A}) whose member admits a Gram matrix with sparsity pattern represented by $B_{\mathscr{A}}^{(\tau)}$, i.e.,
\begin{equation*}
	\Sigma_{\tau}(\mathscr{A}) := \set{f \in \mathscr{P}(\mathscr{A}) | \exists Q \in \mathbb{S}_+^r \left(B_{\mathscr{A}}^{(\tau)}\right)  \; \text{s.t.} \; f = \left(\mathbf{x}^{\mathscr{B}}\right)^T Q \mathbf{x}^{\mathscr{B}} }.
\end{equation*}
Replacing $\Sigma(\mathscr{A})$ with $\Sigma_{\tau}(\mathscr{A})$ in \cref{eq:SOS} yields a sparse SOS relaxation of \cref{eq:P}:

\begin{equation} \label{eq:Ptau}
	(P^{\tau})^*: \quad  \sup_{\lambda} \{\lambda \; | \; f(\mathbf{x}) - \lambda \in \Sigma_{\tau}(\mathscr{A})\}
\end{equation}
with the dual 
\begin{equation} (P^{\tau}): \quad
	\begin{cases} \label{eq:Ptau'}
		\inf & L_{\mathbf{y}}(f)\\
		\text{s.t.} & B_{\mathscr{A}}^{(\tau)} \circ M_{\mathscr{B}}(\mathbf{y}) \succeq 0,\\
		& y_{\mathbf{0}} = 1. 
	\end{cases}
\end{equation}

\begin{proposition}
	Assume that $(P^{\tau})^*$ has a feasible solution. Then $(P^{\tau})^*$ is solvable and there is no duality gap between $(P^{\tau})$ and $(P^{\tau})^*$.
\end{proposition}
\begin{proof}
	Follows from Proposition 3.1 of \cite{Lasserre1} and the fact that each block of $B_{\mathscr{A}}^{(\tau)} \circ M_{\mathscr{B}}(\mathbf{y})$ is a principal submatrix of $M_{\mathscr{B}}(\mathbf{y})$.
\end{proof}

\begin{lemma}
For $k \geq 1$, let $\theta_k$ and $\theta_{\tau}$ with $\tau = k - 1 + \varepsilon$, $\varepsilon \in (0,1)$ be optimal values of $(P^k)$ and $(P^{\tau})$, respectively, then $\theta_{\tau} \leq \theta_k$.
\end{lemma}
\begin{proof}
Since $\supp(B_{\mathscr{A}}^{(\tau)}) \subseteq \supp(B_{\mathscr{A}}^{(k)})$ and $B_{\mathscr{A}}^{(\tau)}$ is block-diagonal, $(P^{\tau})$ is a relaxation of $(P^k)$, therefore $\theta_{\tau} \leq \theta_k$.
\end{proof}

We now explain how to generate a partition $I^{(\tau)}$. For $k \geq 1$, let $\mathscr{S}^{(k-1)}$ be the set defined in \cref{subsec:tssos}. For all $\alpha \in \mathscr{S}^{(k-1)}$, let $A_{\alpha} \in \mathbb{S}^r \cap \mathbb{Z}_2^{r \times r}$ with $r = |\mathscr{B}|$ be a binary matrix indexed by $\mathscr{B}$ such that
\begin{equation}\label{eq:A}
	[A_{\alpha}]_{\beta \gamma} = \begin{cases}
		1, & \text{if } \beta + \gamma = \alpha, \\
		0, & \text{otherwise.}
	\end{cases}
\end{equation}
From \cref{eq:CA} we get $C_{\mathscr{A}}^{(k)} = \sum_{\alpha \in \mathscr{S}^{(k-1)}} A_{\alpha}$ and the amount of elements in $\supp(C_{\mathscr{A}}^{(k)})$ that correspond to the exponent $\alpha  \in \mathscr{S}^{(k-1)}$ is 
\begin{equation}\label{eq:suppA}
	|\supp(A_{\alpha})| = \sum_{\beta \in \mathscr{B}} \sum_{\gamma \in \mathscr{B}} [A_{\alpha}]_{\beta \gamma}.
\end{equation}
By construction $\supp(C_{\mathscr{A}}^{(k)}) \subseteq \supp(B_{\mathscr{A}}^{(k)})$, consequently, $\supp(A_{\alpha}) \subseteq \supp(B_{\mathscr{A}}^{(k)})$ for all $\alpha \in \mathscr{S}^{(k-1)}$. Since the computational cost of solving a block SDP problem $(P^{\tau})^*$ depends on the size of the largest diagonal block in $B_{\mathscr{A}}^{(\tau)}$, in order to minimize the overall runtime we obtain a refinement $I^{(\tau)}$ of $I^{(k)}$ by solving an IP problem that minimizes the width of $I^{(\tau)}$  and restricts $B_{\mathscr{A}}^{(\tau)}$ to satisfy the requirements:
\begin{enumerate}
	\item for all $\alpha \in \mathscr{S}^{(k-1)} \setminus (2\mathscr{B})$, $\supp(B_{\mathscr{A}}^{(\tau)})$ should contain at least $\varepsilon |\supp(A_{\alpha})|$ elements from $\supp(A_{\alpha})$ 
	\item for $\alpha \in 2\mathscr{B}$, $\supp(A_{\alpha}) \subseteq \supp(B_{\mathscr{A}}^{(\tau)})$
\end{enumerate}
Assigning a binary variable to each element of $C_{\mathscr{A}}^{(k)}$ defined in \cref{eq:CA} generally yield a large and, consequently, rather expensive IP problem. That is why we choose to work with a binary matrix $C^{(k)}$ defined in \cref{eq:C}, which normally results in a reduced size IP problem. Furthermore, if the IP problem corresponding to $C^{(k)}$ is too large, it can be replaced with a sequence of smaller IPs as explained in \cref{subsec:alg}. 

\begin{remark}\label{rem1}
	In the constrained case, for the relaxation order $\hat{d}$ and $k \geq 1$, let $I_{j,\hat{d}}^{(k)}$ be the partition of the monomial basis $\mathbb{N}_{\hat{d}-d_j}^n$ induced by a matrix $B_{j,\hat{d}}^{(k)}$, $j = 0, \ldots, m$ defined in \cref{sec:tssos_con}. For parameters $\tau_j = k - 1 + \varepsilon_j$, $j = 0, \ldots, m$ with $\varepsilon_j \in (0,1)$, we aim to find refinements $I_{j,\hat{d}}^{(\tau_j)}$ of partitions $I_{j,\hat{d}}^{(k)}$, where $\varepsilon_j$ controls the width of the partition  $I_{j,\hat{d}}^{(\tau_j)}$.
\end{remark}

\subsection{Algorithm for obtaining $I^{(\tau)}$}

\label{subsec:alg}

In this section, we formulate the algorithm that we use to define a partition $I^{(\tau)}$ of $\mathscr{B}$ with $\tau = k - 1 + \varepsilon$ for $k \geq 1$ and a given parameter value $\varepsilon \in (0,1)$. 

Let us first introduce some necessary notation. Let $P = \set{P_i | i \in [r_p]}$ be a partition of the set $(\mathscr{B})_2$ and $I = \set{I_i | i \in [r_p]}$ with $$I_i = \bigcup_{\delta \in P_i} \mathscr{B}_{\delta} = \bigcup_{\delta \in P_i} \set{\alpha \in \mathscr{B} | (\alpha)_2 = \delta}$$ be the corresponding partition of the monomial basis $\mathscr{B}$. For a partition $P$ and a binary matrix $M \in \mathbb{S}^{r_b} \cap \mathbb{Z}_2^{r_b \times r_b}$ with $r_b = |(\mathscr{B})_2|$ indexed by $(\mathscr{B})_2$, we define a binary matrix $D_M \in \mathbb{S}^{r_p} \cap \mathbb{Z}_2^{r_p \times r_p}$ by
\begin{equation}\label{eq:DM}
	[D_M]_{ij} = \begin{cases}
		1, & \text{if } 1 \in \{M_{\delta \sigma} : \delta \in P_i, \sigma \in P_j\}, \\
		0, & \text{otherwise.} 
	\end{cases}
\end{equation} 
For a vector $\nu \in (\mathscr{B} + \mathscr{B})_2$, we define a binary matrix $E_{\nu} \in \mathbb{S}^{r_b} \cap \mathbb{Z}_2^{r_b \times r_b}$ with $r_b = |(\mathscr{B})_2|$ indexed by $(\mathscr{B})_2$ as
\begin{equation}\label{eq:E}
	[E_{\nu}]_{\delta \sigma} = \begin{cases}
		1, & \text{if } \delta + \sigma = \nu, \\
		0, & \text{otherwise.}
	\end{cases}
\end{equation}

For $k \geq 1$ and a chosen parameter value $\varepsilon \in (0,1)$, a partition $I^{(\tau)}$ of $\mathscr{B}$ is defined using \cref{alg:main}.

\begin{algorithm}
	\caption{Defining a partition $I^{(\tau)}$ with $\tau = k - 1 + \varepsilon$}
	\label{alg:main}
	\begin{algorithmic}[1]
\STATE\label{line1}{Choose $k\geq 1$ and $\varepsilon \in (0,1)$}
\STATE{Define $C^{(k)}$ using \cref{eq:C}} 
\STATE{Set $W = C^{(k)}$} 
\STATE{Set $P = \{\left(\mathscr{B}\right)_2\}$, i.e., the trivial partition} 
\STATE\label{line5}{Set $\mathscr{S} = \mathscr{S}^{(k-1)}$ with $\mathscr{S}^{(k-1)}$ defined in \cref{subsec:tssos}} 
\FOR{$\nu \in \left(\mathscr{S}\right)_2 \setminus \{\mathbf{0}\}$}
\STATE{Define $E_{\nu}$ using \cref{eq:E}}
\STATE{Compute $M = W \circ (E_{\nu} + E_{\mathbf{0}})$}
\STATE{Compute $D_M$ based on $M$ and $P$ using \cref{eq:DM}}
\STATE\label{line10}{Update $P$ solving an IP problem for $\varepsilon$, $\nu$ and $D_M$ defined in \cref{subsec:assembling}}
\ENDFOR
\STATE{Set $I^{(\tau)}$ to the partition of $\mathscr{B}$ corresponding to the resulting partition $P$ of $(\mathscr{B})_2$}
\end{algorithmic}
 \end{algorithm}
Instead of iterating through $\nu \in \left(\mathscr{S}^{(k-1)}\right)_2 \setminus \{\mathbf{0}\}$ in \cref{alg:main} we could also formulate a single Integer Programming problem. However, this IP problem might be rather expensive to solve. Replacing it with a sequence of smaller IP problems normally results in a significant reduction of the computing time.

\begin{remark} \label{rem:chordal_refined}
	The new approach can also be used within the chordal-TSSOS. Let $B_{\mathscr{A}}^{(\tau)}$ be the block-diagonal matrix corresponding to the partition $I^{(\tau)}$ obtained using \cref{alg:main} and let $C_{\mathscr{A}}^{(k)}$, $k \geq 1$ be a binary matrix defined in \cref{eq:CA}. Maximal cliques of the graph obtained by applying a chordal-extension operation to the adjacency graph of $C_{\mathscr{A}}^{(k)} \circ B_{\mathscr{A}}^{(\tau)}$ might produce a cheaper SDP relaxation than the one generated by the $k$-th iterative step of the chordal-TSSOS. We further refer to the resulting method as the refined chordal-TSSOS.
\end{remark}

\begin{remark}
	In the constrained case, for the relaxation order $\hat{d}$ and parameter $\tau_j = k - 1 + \varepsilon_j$ with $k \geq 1$ and $\varepsilon_j \in (0,1)$, $j = 0, \ldots, m$  a partition $I_{j,\hat{d}}^{(\tau_j)}$ of $\mathscr{B}^{(j)} = \mathbb{N}^n_{\hat{d}-d_j}$ is obtained by applying \cref{alg:main} to $C^{(j,k)}$ defined in \cref{eq:CC} with  $\mathscr{S} = \mathscr{S}^{(k-1)}$ being replaced with $\mathscr{S}= \left(\mathscr{S}_{0,\hat{d}}^{(k-1)} - \supp(g_j)\right) \cap \left(\mathscr{B}^{(j)}+ \mathscr{B}^{(j)}\right)$, where $\mathscr{S}_{0,\hat{d}}^{(k-1)}$ is defined in \cref{subsec:tssos}.
\end{remark}

\subsection{Assembling an IP problem}

\label{subsec:assembling}

We now define an IP problem used in \cref{alg:main}. Let $\mathscr{S}$ be the set defined in \cref{line5} of \cref{alg:main} and let $\mathscr{S}_{\nu}$ be the subset of $\mathscr{S}$ containing elements with a parity type $\nu \in \left(\mathscr{S}\right)_2$, i.e., $\mathscr{S}_{\nu} := \set{\beta \in \mathscr{S} | (\beta)_2 = \nu}$. Let $P = \set{P_i | i \in [r_p]}$ be the current partition of the set $(\mathscr{B})_2$ with $I = \set{I_i | i \in [r_p]}$ being the corresponding partition of $\mathscr{B}$. Solving an Integer Programming problem for $\nu \in \left(\mathscr{S}\right)_2 \setminus \{\mathbf{0}\}$ we aim to find a new partition $P'$ of $\left(\mathscr{B}\right)_2$ with $P \leq P'$ such that the corresponding partition $I'$ of $\mathscr{B}$ satisfies the following condition further referred to as \textbf{C.1}: for all $\alpha \in \mathscr{S}_{\nu}$, the support of the block-diagonal matrix inducing $I'$ contains at least $\varepsilon |\supp(A_{\alpha})|$ elements from $\supp(A_{\alpha})$, where $\varepsilon$ is a parameter specified in \cref{line1} of \cref{alg:main} and the matrix $A_{\alpha}$ is defined in \cref{eq:A}. In order to approximate the goal of minimizing the overall computational cost of solving $(P^{\tau})$, we optimize over the set of all partitions $I'$ satisfying condition \textbf{C.1} with the objective function being the width of $I'$.

Let $Y$ be a matrix variable defined using the binary matrix $\overline{D_M} \in \mathbb{S}^{r_p} \cap \mathbb{Z}_2^{r_p \times r_p}$ by
\begin{equation} \label{eq:Y}
	Y_{ij} = \begin{cases}
		1, & \text{if } $i = j$, \\
		y_{ij} \in \{0,1\}, & \text{if } [\overline{D_M}]_{ij} = 1 \text{ and } i \neq j, \\
		0, & \text{otherwise}
	\end{cases}
\end{equation}
and let $G^d := \set{G_s^d | s \in [r_d]}$ be the set of connected components of the adjacency graph of $\overline{D_M}$. For $\alpha \in \mathscr{S}_{\nu}$, we define a matrix $K_{\alpha} \in \mathbb{S}^{r_p} \cap \mathbb{Z}^{r_p \times r_p}$ by
\begin{equation}\label{eq:K}
	[K_{\alpha}]_{ij} =  \sum_{\beta \in I_i} \sum_{\gamma \in I_j} [A_{\alpha}]_{\beta \gamma}.
\end{equation} 
Using this notation we can rewrite \cref{eq:suppA} as
\begin{equation}\label{eq:suppA2}
	|\supp(A_{\alpha})| = \sum_{\beta \in \mathscr{B}} \sum_{\gamma \in \mathscr{B}} [A_{\alpha}]_{\beta \gamma} = \sum_{i = 1}^{r_p} \sum_{j = 1}^{r_p}  [K_{\alpha}]_{ij}
\end{equation}
The new partition $P'$ is obtained by solving the following IP problem:
\begin{mini*}|l|
	{\omega \in \mathbb{Z},Y \text{ as in } \cref{eq:Y}}{\omega}
	{}{}
	\addConstraint{Y_{ik} + Y_{kj} - Y_{ij} \leq 1, \underset{\text{($i \neq j$, $j \neq k$, $k \neq i$)}}{\; \forall \; i,j,k \in G_s^d, \; s \in [r_d]} }
	\addConstraint{\sum_{k = 1}^{r_p} |I_k| Y_{ik} \leq \omega, \; i \in [r_p]}
	\addConstraint{\sum_{i = 1}^{r_p} \sum_{j = 1}^{r_p} [K_{\alpha}]_{ij} Y_{ij} \geq \varepsilon \sum_{i = 1}^{r_p} \sum_{j = 1}^{r_p} [K_{\alpha}]_{ij}, \; \forall \; \alpha \in \mathscr{S}_{\nu}},
\end{mini*}
where $\omega$ is an integer variable corresponding to the width of $I'$, the first set of constraints further refered to as (C1) restricts $Y$ to be block-diagonal, the second set of constraints (C2) guarantees that the width of $I'$ does not exceed $\omega$, and the third set of constraints (C3) is imposed to make $I'$ satisfy condition \textbf{C.1}.

 Let $G^y := \set{G_s^y | s \in [r_y]}$ be the set of connected components of the adjacency graph of $Y$. If $P = \{P_i \; | \; i \in [r_p]\}$ is the current partition of $\left(\mathscr{B}\right)_2$, then the new partition $P'$ is given by $P' = \set{P'_s | s \in [r_y]}$ with $P'_s = \cup_{i \in G_s^y} P_i$.

\subsection{Example: refined TSSOS}

\begin{example} \label{ex:ex3}
	Let us now apply \cref{alg:main} to the polynomial from \cref{ex:ex2}, for which we have $\left(\mathscr{S}^{(0)}\right)_2 = \set{(0,0,0),(1,0,0),(0,1,0),(1,0,1)}$ with
	\begin{itemize}
		\item $\mathscr{S}_{(1,0,0)}^{(0)} = \set{(1,2,2),(1,0,2)}$
		\item $\mathscr{S}_{(0,1,0)}^{(0)} = \set{(0,3,2)}$ 
		\item $\mathscr{S}_{(1,0,1)}^{(0)} = \set{(1,0,3)}$
	\end{itemize}
 Let us choose $k = 1$ and $\varepsilon = 0.2$. The partition $P$ of $\left(\mathscr{B}\right)_2$ is initially set to the trivial partition, i.e., $P = \set{\left(\mathscr{B}\right)_2}$. After solving IP problems for $\nu = (1,0,0)$ and $\nu = (0,1,0)$, we get the partition $P = \set{P_1, P_2, P_3, P_4, P_5}$ with
\begin{itemize}
	\item $P_1 = \set{(0,0,0)}$
	\item $P_2 = \set{(1,0,0)}$
	\item $P_3 = \set{(0,1,0)}$
	\item $P_4 = \set{(0,0,1),(1,0,1),(0,1,1),(1,1,1)}$
	\item $P_5 = \set{(1,1,0)}$
\end{itemize}

\vspace{5pt}

For $\nu = (1,0,1)$, using \cref{eq:E} and $M = C^{(1)} \circ (E_{101} + E_{000})$ we get
\begin{equation*}
	E_{101} = \begin{bmatrix}
		0 & 0 & 0 & 0 & 0 & 1 & 0 & 0 \\
		0 & 0 & 0 & 1 & 0 & 0 & 0 & 0 \\
		0 & 0 & 0 & 0 & 0 & 0 & 0 & 1 \\
		0 & 1 & 0 & 0 & 0 & 0 & 0 & 0 \\ 
		0 & 0 & 0 & 0 & 0 & 0 & 1 & 0 \\
		1 & 0 & 0 & 0 & 0 & 0 & 0 & 0 \\
		0 & 0 & 0 & 0 & 1 & 0 & 0 & 0 \\
		0 & 0 & 1 & 0 & 0 & 0 & 0 & 0  
	\end{bmatrix} 
 \quad 
M  =  \begin{bmatrix}
	1 & 0 & 0 &   0 & 0 & 1 &  0 &  0 \\
	0 & 1 & 0 & 1 & 0 & 0 & 0 & 0 \\
	0 & 0 & 1 & 0 & 0 & 0 & 0 & 0 \\
	0 & 1 & 0 &  1 & 0 & 0 & 0 & 0 \\ 
	0 & 0 & 0 & 0 & 1 & 0 & 0 & 0 \\
	1 & 0 & 0 & 0 & 0 & 1 & 0 & 0 \\
	0 & 0 & 0 & 0 & 0 & 0 & 1 & 0\\
	0 & 0 & 0 & 0 & 0 & 0 & 0 & 1
\end{bmatrix},
\end{equation*}
where the matrices $E_{101}$ and $M$ are indexed by $$(\mathscr{B})_2 = \{(0,0,0), (1,0,0), (0,1,0), (0,0,1), (1,1,0), (1,0,1), (0,1,1), (1,1,1)\}.$$

We now obtain the binary matrix $D_M \in \mathbb{S}^{r_p} \cap \mathbb{Z}_2^{r_p \times r_p}$ with $r_p = |P| = 5$ by applying \cref{eq:DM} to the matrix $M$ and the current partition $P$. We proceed in the following way: $[D_M]_{24} = 1$ since $1 \in \set{[M]_{\delta \sigma} | \delta \in P_2,  \sigma \in P_4}$ and $[D_M]_{34} = 0$ since $1 \notin \set{[M]_{\delta \sigma} | \delta \in P_3, \sigma \in P_4}$. Repeating this process for all $(i,j) \in [5]\times[5]$ we get
\begin{equation*}
	D_M = \begin{bmatrix}
		1 & 0 & 0 &  1 & 0 \\
		0 & 1 & 0 &  1 & 0 \\ 
		0 & 0 & 1 & 0 & 0 \\ 
		1 & 1 & 0 & 1 & 0 \\ 		
		0 & 0 & 0 & 0 & 1
	\end{bmatrix}  \quad \text{and}\quad
\overline{D_M} = \begin{bmatrix}
	1 & 1 & 0 & 1 & 0 \\
	1 & 1 & 0 & 1 & 0 \\ 
	0 & 0 & 1 & 0 & 0 \\ 
	1 & 1 & 0 & 1 & 0 \\ 		
	0 & 0 & 0 & 0 & 1
\end{bmatrix}
\end{equation*}
with the set of connected components of the adjacency graph of $\overline{D_M}$ being $G^d = \set{\set{1,2,4},\set{3},\set{5}}$. Using \cref{eq:Y} we defined the variable matrix
\begin{equation*}
	Y = \begin{bmatrix}
		1 & y_{12} & 0 & y_{14} & 0 \\
		y_{12} & 1 & 0 & y_{24} & 0 \\ 
		0 & 0 & 1 & 0 & 0 \\ 
		y_{14} & y_{24} & 0 & 1 & 0 \\ 		
		0 & 0 & 0 & 0 & 1
	\end{bmatrix},
\end{equation*}
producing to the following sets of (C1) and (C2) constraints:
\begin{align}  \label{eq:C12}
	-y_{12} + y_{14} + y_{24} & \leq 1   &   |I_1| \cdot 1 + |I_2| \cdot y_{12} + |I_4| \cdot y_{14} & \leq \omega \nonumber\\
	 y_{12} - y_{14} + y_{24} & \leq 1    &   |I_1| \cdot y_{12} + |I_2| \cdot 1 + |I_4| \cdot y_{24} & \leq \omega \\ 
	 y_{12} + y_{14} - y_{24} & \leq 1    &   |I_1| \cdot y_{14} + |I_2| \cdot y_{24} + |I_4| \cdot1 & \leq \omega  \nonumber
\end{align}  
with $|I_1| = 4$, $|I_2| = 4$, $|I_4| = 7$, where we use $|I_k| = \sum_{\delta \in P_k} |\mathscr{B}_{\delta}|$. Since
\begin{equation*}
	K_{103} = \begin{bmatrix}
		0 & 0 & 0 & 1 & 0 \\
		0 & 0 & 0 & 2 & 0 \\
		0 & 0 & 0 & 0 & 0 \\
		1 & 2 & 0 & 0 & 0 \\
		0 & 0 & 0 & 0 & 0
	\end{bmatrix}
\end{equation*} the corresponding (C3) constraint is
\begin{equation}  \label{eq:C3}
	y_{14} + 2y_{24} \geq \varepsilon \cdot 3.
\end{equation}

Minimizing $\omega \in \mathbb{Z}$ subject to constraints \cref{eq:C12} and \cref{eq:C3}, we get: $y_{12} = y_{24} = 0$, $y_{14} = 1$. This results in the partition $P = \set{P_1, P_2, P_3, P_4}$ with
\begin{itemize}
	\item $P_1 = \set{(0,0,0),(0,0,1),(1,0,1),(0,1,1),(1,1,1)}$,
	\item $P_2 = \set{(1,0,0)}$,
	\item $P_3 = \set{(0,1,0)}$,
	\item $P_4 = \set{(1,1,0)}$,
\end{itemize}
giving the following partition $I^{(0.2)}$ of the monomial basis $\mathscr{B}$:
\begin{enumerate} [label={\arabic*)}]
	\item $\set{1, x_1^2, x_2^2, x_3^2, x_3, x_3^3, x_1^2 x_3, x_2^2 x_3, x_1 x_3, x_2 x_3, x_1 x_2 x_3}$,
	\item $\set{x_1, x_1^3, x_1 x_2^2, x_1 x_3^3}$,
	\item $\set{x_2, x_2^3, x_1^2 x_2, x_2 x_3^2}$,
	\item $\set{x_1 x_2}$.
\end{enumerate}
Solving the corresponding SDP problem we obtain $\theta_{0.2} \approx -29.6934$.
\end{example}

\section{Numerical experiments}
\label{sec:num_res}
 In this section, we present numerical results for the proposed sparse moment-SOS relaxations constructed with $k = 1$ and different parameter values $\varepsilon \in (0,1)$ for both unconstrained and constrained polynomial optimization problems. Since $\tau = k - 1 + \varepsilon$, we have $\tau \in (0,1)$. Our algorithm, named rTSSOS (refined TSSOS) is implemented in Julia, utilizes JuMP \cite{JuMP} to model IP problems from \Cref{alg:main} and relies on MOSEK \cite{Mosek} to solve them. The corresponding SDP problems are assembled and solved using functions from the TSSOS tool \cite{TSSOStool}. In the following subsections, we compare the performance of rTSSOS with that of the block and chordal-TSSOS methods. The block TSSOS method is described in \cref{subsec:tssos} and the chordal-TSSOS is obtained by replacing the block-closure operation from \cref{eq:block_closure} with a chordal-extension operation on the adjacency graph. The numerical results for the block and chordal-TSSOS were obtained using the TSSOS tool.

All numerical examples were computed on a server with the Linux system. The timing of rTSSOS includes the time for pre-processing (to get all the necessary data for IP and SDP problems) and the time for assembling and solving IPs and SDP. The  runtime for all methods was obtained using the \verb|@elapsed| function. In our computations we set the CPU time limit for the SDP solver to 5000 seconds. If for the relaxation produced by some method the SDP solver terminates with the  status different from \verb|OPTIMAL|,  the problem is considered to be unsolved by this method. Note that the time the solver actually spent on the problem was in several cases significantly larger than the imposed upper limit of 5000s due to solver-specific reasons. In our statistics, we used the values that were reported by the \verb|@elapsed| function. The notations used in this section are listed in \cref{tab:notation}.

\begin{table}[tbhp]
	\footnotesize
	\captionsetup{position=top} %<- Needed for using subtables created with the subfig package
	\caption{Notation}\label{tab:notation} 
	\begin{center}
		\begin{tabular}{|c|c|l|} \hline
			\multicolumn{2}{|c|}{$n$}  & the number of variables \\  \hline
			\multicolumn{2}{|c|}{$2d$} & the degree of a polynomial \\  \hline
			\multicolumn{2}{|c|}{$s$} & the length of a support   \\  \hline
			\multicolumn{2}{|c|}{$k$} & the iterative step of the TSSOS (chordal or block)\\ \hline
			\multicolumn{2}{|c|}{$\hat{d}$} & the ralaxation order of Lasserre hierarchy \\  \hline
			\multicolumn{2}{|c|}{$mb$} & the maximal size of blocks \\ \hline
			\multicolumn{2}{|c|}{$N_u$} & the amount of unsolved problems \\ \hline
			\multicolumn{2}{|c|}{$\Theta_M$} & the optimal value  obtained with method $M$ \\ \hline
			\multicolumn{2}{|c|}{$\Theta_B$} & the best optimal value over all methods \\ \hline
			\multicolumn{2}{|c|}{$T_M$} & the computing time  for method $M$\\  \hline
			\multirow{4}{*}{$M$:} & $\tau$ & the rTSSOS for the parameter value $\tau$  \\ \cline{2-3}
			& $t1$ & the block TSSOS ($k = 1$) \\ \cline{2-3}
			& $c1$ & the chordal-TSSOS ($k = 1$) \\ \cline{2-3}
			& $c2$ & the chordal-TSSOS  ($k = 2$)\\ \cline{2-3}	\hline
		\end{tabular}
	\end{center} 
\end{table}

\subsection{Unconstrained polynomial optimization problems} 

\label{subsec:res_uncon}

\begin{example} \label{ex:illustrative}
	Let us start with an illustrative example and consider the following polynomials:
	$$f_1 = \sum_{i = 1}^{8} \frac{1}{i} x_{i}^{8} + \sum_{i = 1}^{4} (-1)^{i} \frac{1}{2i} x_i^4 x_{i+2} x_{i+4} + \sum_{i = 1}^{5} (-1)^{i} \frac{1}{3i} x_i x_{i+1} x_{i+2}^2 x_{i+3},$$ 
	$$f_2 = \sum_{i = 1}^{8} \frac{1}{i} x_{i}^{8} + \sum_{i = 1}^{4} (-1)^{i} \frac{1}{2i} x_i^2 x_{i+2} x_{i+4} + \sum_{i = 1}^{5} (-1)^{i} \frac{1}{3i} x_i x_{i+1}^2 x_{i+2}^2 x_{i+3},$$ 
	$$f_3 = \sum_{i = 1}^{8} \frac{1}{i} x_{i}^{8} + \sum_{i = 1}^{4} (-1)^{i} \frac{1}{2i} x_i^2 x_{i+2} + \sum_{i = 1}^{5} (-1)^{i} \frac{1}{3i} x_i x_{i+1} x_{i+2}^2.$$
	The polynomials $f_1$, $f_2$ and $f_3$ have 8 variables and are of degree 8. The monomial basis is $\mathbf{x}^{\mathbb{N}_4^8}$. The numerical results on these polynomials listed in \cref{tab:illustrative} demonstrate a potential speed-up of the refined TSSOS compared to the block TSSOS. \cref{tab:illustrative} also shows that for a polynomial $f \in \mathbb{R}[\mathbf{x}]$ and a fixed parameter value $\tau$, the ratio of the maximal size of blocks in the relaxation generated by the refined TSSOS to the maximal size of blocks in the relaxation given by the first iterative step of the block TSSOS method ($mb_{\tau}/mb_{t1}$) strongly depends on the support of $f$. 
	
\begin{table}[tbhp]
	\footnotesize
	\captionsetup{position=top} %<- Needed for using subtables created with the subfig package
	\caption{Numerical results for the polynomials from \cref{ex:illustrative} demonstrate a potential speed-up of the refined TSSOS compared to the block TSSOS and dependence of $mb_{\tau}/mb_{t1}$ on the support of $f$.}\label{tab:illustrative} 
	\begin{center}
		\begin{tabular}{cc|c|c|c|c|c|c|c|c|} \cline{3-10}
			&  & \multirow{2}{*}{block TSSOS} & \multicolumn{7}{c|}{refined TSSOS} \\ \cline{4-10}
			& & & 0.1 & 0.2 & 0.3 & 0.4 & 0.5 & 0.6 & 0.7 \\ \hline
			\multicolumn{1}{ |c  }{\multirow{3}{*}{$f_1$ }} & \multicolumn{1}{ |c| }{mb} & 340 & 45 & 45 & 45 & 45 & 45 & 95 & 197 \\ \cline{2-10}
			\multicolumn{1}{ |c  }{}  &  \multicolumn{1}{ |c| }{time}  & 49.874 & 1.222 & 1.063 & 1.138  & 1.048 & 1.051 & 2.421 & 20.897 \\  \cline{2-10}
			\multicolumn{1}{ |c  }{}  &  \multicolumn{1}{ |c| }{opt}  & -0.132 & -0.132 & -0.132 & -0.132 & -0.132 & -0.132 & -0.132 & -0.132 \\ \hline
			\multicolumn{1}{ |c  }{\multirow{3}{*}{$f_2$ }} & \multicolumn{1}{ |c| }{mb} & 248 & 45 & 46 & 46 & 63 & 63 & 65 & 185 \\ \cline{2-10}
			\multicolumn{1}{ |c  }{}   &  \multicolumn{1}{ |c| }{time}  & 24.699 & 0.793 & 0.931 & 0.881 & 1.306 & 1.270 & 1.415 & 14.263 \\ \cline{2-10}
			\multicolumn{1}{ |c  }{}  &  \multicolumn{1}{ |c| }{opt}  & -0.220 & -0.246 & -0.246 & -0.246 & -0.220 & -0.220 & -0.220 & -0.220 \\ \hline
			\multicolumn{1}{ |c  }{\multirow{3}{*}{$f_3$ }} & \multicolumn{1}{ |c| }{mb} & 184 & 54 & 54 & 54 & 100 & 100 & 102 & 184 \\ \cline{2-10}
			\multicolumn{1}{ |c  }{}   &  \multicolumn{1}{ |c| }{time}  & 10.425 & 1.033 & 0.944 & 0.979 & 2.629 & 2.446 & 2.758 & 9.763 \\ \cline{2-10}
			\multicolumn{1}{ |c  }{}  &  \multicolumn{1}{ |c| }{opt}  & -0.385 & -0.464 & -0.464 & -0.464 & -0.385 & -0.385 & -0.385 &-0.385 \\ \hline
		\end{tabular}
	\end{center} 
\end{table}	
\end{example}

We now test our approach on three sets of random polynomials. Polynomials from the first set (set I) are defined by
\begin{equation*}
	f = c_0 + \sum_{i = 1}^{n} c_i x_i^{2d} + \sum_{j = 1}^{s - n - 1} c_j' \mathbf{x}^{\mathbf{\alpha}_j} \in \text{\textbf{randpolyI}}(n,2d,s)
\end{equation*}
and constructed as follows: we randomly choose coefficients $c_i$ between $0$ and $1$ as well as $s-n-1$ vectors $\mathbf{\alpha}_j$ in $\mathbb{N}_{2d-1}^{n} \setminus \{0\}$ with random coefficients $c_j'$ between $-1$ and $1$. The second set (set II) consists of polynomials $f \in \text{\textbf{randpolyI}}(n,2d,s)$ that satisfy an additional requirement, namely, we restrict all vectors $\alpha_j$, $j \in \set{1, \ldots, s-n-1}$ to have at least 5 non-zero components. Polynomials $f \in \text{\textbf{randpolyI}}(n,2d,s)$ from the third set (set III) have $\alpha_j$, $j \in \set{1, \ldots, s-n-1}$ with at most 3 non-zero components. To form each of these sets we use polynomials from 10 different classes:
\begin{center}
	\begin{tabular}{r|c|c|c|c|c|c|c|c|c|c|} 
		\multicolumn{1}{r}{} & \multicolumn{1}{c}{1} & \multicolumn{1}{c}{2} & \multicolumn{1}{c}{3} & \multicolumn{1}{c}{4} & \multicolumn{1}{c}{5} & \multicolumn{1}{c}{6} & \multicolumn{1}{c}{7} &
		\multicolumn{1}{c}{8} & \multicolumn{1}{c}{9} & \multicolumn{1}{c}{10} \\ \cline{2-11}		
		$n$ & 8 & 8 & 8 & 9 & 9 & 9 & 10 & 10 & 11 & 11 \\
		$2d$ & 8 & 8 & 8 & 8 & 8 & 8 & 8 & 8 & 8 & 8 \\
		$s$ & 17 & 19 & 21 & 19 & 21 & 23 & 21 & 23 & 23 & 25 \\\cline{2-11}	
	\end{tabular}
\end{center} \vspace{5pt}
For each of these classes we generate 100 random polynomials. The information about the IP problems solved to obtain rTSSOS relaxations is presented in \cref{tab:ip_sets}. Numerical results for the refined, block and chordal-TSSOS  on polynomials from sets I, II and III are presented in rows 1, 2 and 3 of \cref{fig:uncon} as well as in \cref{tab:num1,tab:num5,tab:num7} in \cref{sec:appendix}. Comparing the runtime from \cref{tab:ip_sets,tab:num1,tab:num5,tab:num7} one can see that solving of IP problems takes only a small fraction of the overall computational cost. The distribution of the maximal size of blocks for the refined and block TSSOS on polynomials from set I is depicted in \cref{fig:block_size_refined}.

 \begin{figure}[tbhp]
 	\centering 
 	\includegraphics[trim = 110 50 150 50, clip, width=1.0\linewidth]{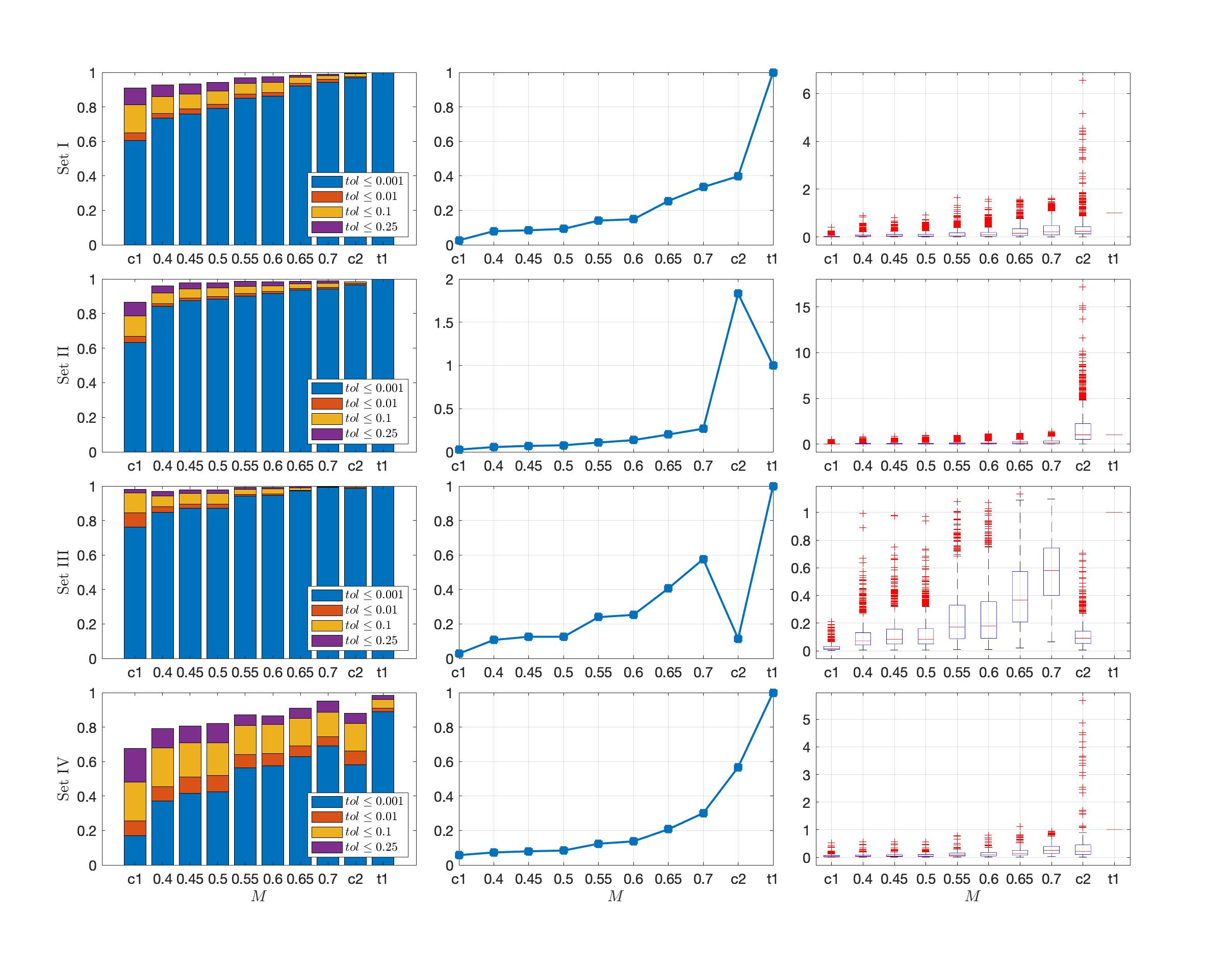}
 	\vspace{-20pt}
 	\caption{Numerical results for the chordal ($c1$ and $c2$), block ($t1$) and refined TSSOS with different parameter values $\tau \in (0, 1)$ in the unconstrained case: plots in column 1 show the fraction of problems solved to a certain accuracy by method $M$, i.e., with $\frac{|\Theta_{M} - \Theta_B|}{|\Theta_B|} \leq tol$, the ratio of CPU time for method $M$ to CPU time of TSSOS ($k=1$), i.e., $T_{M} / T_{t1}$, averaged over all problems in the set is depicted in column 2, box plots showing the spread of values $T_{M} / T_{t1}$ are given in column 3, where the central mark indicates the median, and the bottom and top edges of the box indicate the 25th and 75th percentiles, respectively. Note that the CPU time limit for the SDP solver is 5000 seconds. If for some problem the solver fails to terminate within this time on a relaxation produced by method $M$, this problem is considered to be unsolved by this method.}
 	\label{fig:uncon}
 \end{figure}

 \begin{figure}[tbhp]
	\centering 
	\includegraphics[trim = 115 20 115 50, clip, width=0.8\linewidth]{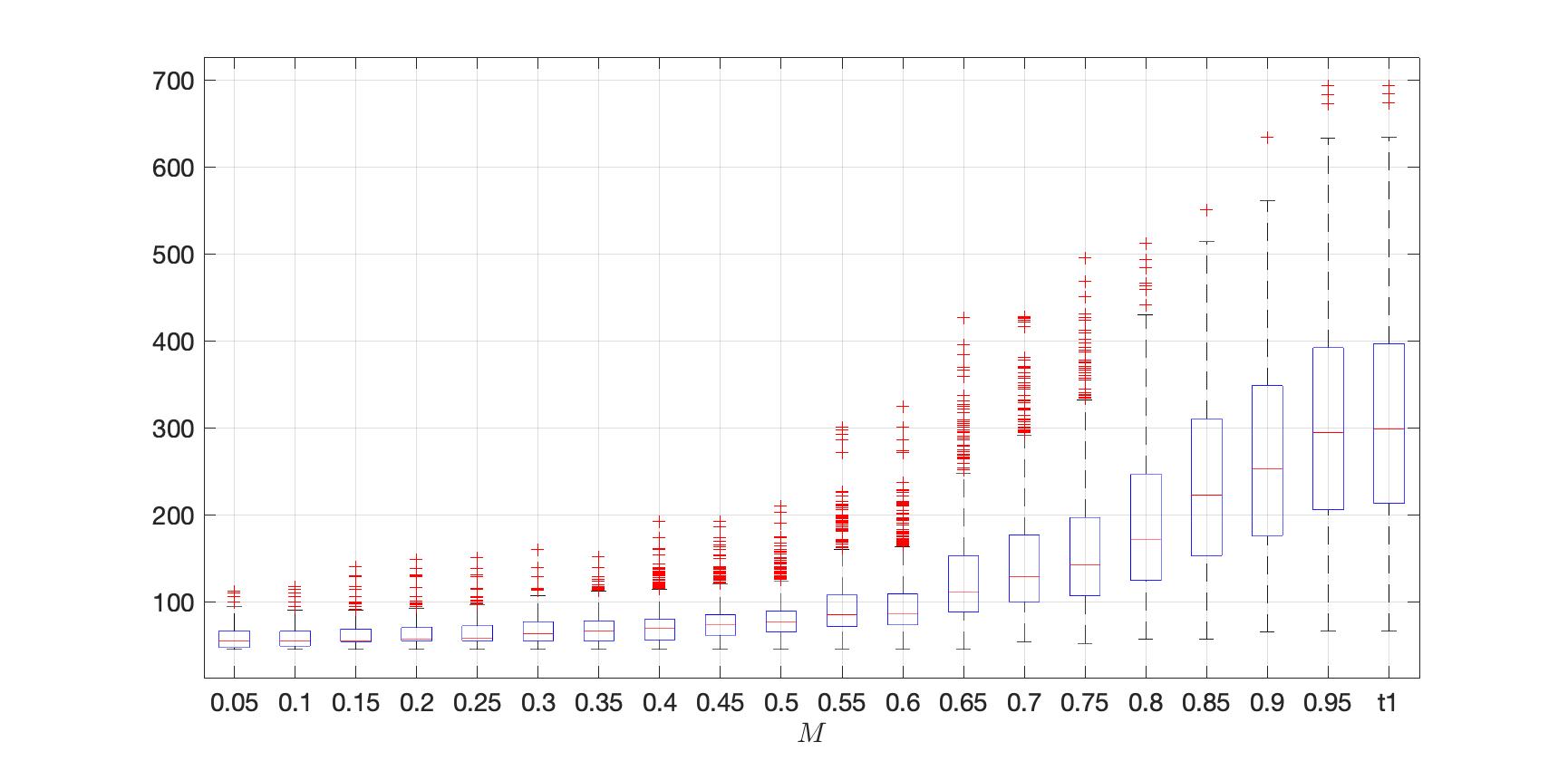}
	\caption{The distribution of the maximal size of blocks for the block ($t1$) and refined TSSOS with different parameter values $\tau \in (0, 1)$ on polynomials from set I: the central mark indicates the median, and the bottom and top edges of the box indicate the 25th and 75th percentiles, respectively. The whiskers extend to the most extreme data points not considered outliers, and the outliers are plotted individually using the '+' symbol.}
	\label{fig:block_size_refined}
\end{figure}
 
 The obtained numerical results show that on polynomials from sets I, II and III the refined TSSOS method allows to obtain cheaper relaxations than the ones returned by the first iterative step of the block TSSOS, rTSSOS with a well-chosen parameter value $\tau \in (0,1)$ also provides more accurate bounds than the first step of the chordal-TSSOS with only moderate increase in the computing time. It is still an open question how to choose an optimal parameter value $\tau$ that would allow to obtain a good quality bound in a reasonable time. On set II, the refined TSSOS performs much better than the second iterative step of the chordal-TSSOS, which becomes very expensive on this kind of polynomials. On the other hand, on set III the second iterative step of the chordal-TSSOS is significantly faster than rTSSOS with large $\tau$ and more accurate than rTSSOS with small $\tau$.

Newton polytopes of polynomials from sets I, II and III are scaled standard simplices. To test our approach on polynomials with more general Newton polytopes we use polynomials $ f \in \text{\textbf{randpolyII}}(n,2d,k_1,k_2,k_3,k_4)$ generated as follows:
\begin{enumerate}
	\item Randomly generate a partition $\set{A_1, A_2, A_3}$ of the set $\set{1, \ldots, n}$ with $|A_2| + |A_3| = k_1$.
	\item Randomly pick vectors $\gamma_i$, $i \in \set{1, \ldots, k_2}$ from $2 \mathbb{N}_{(d+2)}^n \setminus 2\mathbb{N}_{d}^n$.
	\item Define a polynomial $g = \sum_{i \in A_1} c_i x_i^{2d} + \sum_{i \in A_2} c_i x_i^{2(d+1)} + \sum_{i \in A_3} c_i x_i^{2(d+2)} + \sum_{j = 1}^{k_2} c_i^e \mathbf{x}^{\gamma_i}$ with random coefficients $c_i$ and $c_i^e$ between 0 and 1.
	\item Set $d_g = \deg(g)/2$ and $\mathscr{B}_g = \frac{1}{2} \cdot \New(g) \cap \mathbb{N}^n$.
	\item Define $f = g + \sum_{j = 1}^{k_3} c_j' \mathbf{x}^{\alpha_j} + \sum_{j = 1}^{k_4} c_j^{o} \mathbf{x}^{\beta_j}$, where we randomly choose $k_3$ vectors $\alpha_j$ in $\mathbb{N}_{2d}^n$ with coefficients $c_j'$ between -1 and 1 as well as $k_4$ vectors $\beta_j$ in $\left(\mathscr{B}_g + \mathscr{B}_g \right) \setminus \left(\mathbb{N}_{2d}^n \cup 2\mathscr{B}_g \right)$ with $c_j^o$ between 0 and 1 if $|\beta_j| = 2 d_g$ and between -1 and 1 if  $|\beta_j| < 2 d_g$.
\end{enumerate}

We consider polynomials from two classes: $$(n,2d,k_1,k_2,k_3,k_4) \in \{(8,8,2,4,8,4), (8,8,4,6,8,4)\}.$$ 
For each of these classes we generate 100 random polynomials. We refer to these polynomials as set IV.
 Numerical results on this set are  displayed in row 4 of \cref{fig:uncon} and listed in \cref{tab:num3} in \cref{sec:appendix}. Note that the time for computing a monomial basis is included in the time of the first iterative step of the block and chordal-TSSOS as well as the refined TSSOS for all $\tau$. 

Finally, we generate 20 random polynomials $H_1, \ldots, H_{10} \in \text{\textbf{randpolyI}}(10,4,30)$, $H_{11}, \ldots, H_{20} \in \text{\textbf{randpolyI}}(10,4,35)$, such that all vectors $\alpha_j$, $j \in \set{1, \ldots, s-n-1}$ have at least 6 non-zero components. Numerical results on polynomials $H_1, \ldots, H_{20}$  listed in \cref{num10} demonstrate that in comparison with the first iterative step of the chordal-TSSOS method application of the refined TSSOS method to this kind of polynomials either results in cheaper relaxations (possibly sacrificing the accuracy) or in better quality bounds (possibly increasing the computational costs). 

\begin{table}[tbhp]
	\footnotesize
	\captionsetup{position=top} %<- Needed for using subtables created with the subfig package
	\caption{Numerical results on polynomials $H_i$, $i = 1, \ldots, 20$: for these polynomials the refined TSSOS with different parameter values $\tau$ results either in a computational speed-up or in better quality bounds compared to the chordal-TSSOS ($k = 1$)  }\label{num10} 
	\begin{center}
		\begin{tabular}{|c|c|c|c|c|c|c|c|c|}  \hline
			& \multicolumn{2}{|c|}{chordal-TSSOS } & \multicolumn{6}{|c|}{ refined TSSOS} \\  \cline{4-9}
			&\multicolumn{2}{|c|}{$k = 1$} &  \multicolumn{2}{|c|}{$\tau = 0.25$} &  \multicolumn{2}{|c|}{$\tau = 0.3$}  &  \multicolumn{2}{|c|}{$\tau = 0.35$} \\ \cline{2-9}
			& $Opt$ & $Time$ & $Opt$ & $Time$ & $Opt$ & $Time$ & $Opt$ & $Time$   \\ \hline
			$H_1$ & -1634.37 & 32.47 & -1720.11  & 3.58 & -1634.36 & 4.31 & -1634.37 & 9.81 \\
			$H_2$ & -4636.21 & 32.16 & -5568.03 & 3.46 & -4636.19 & 7.24 & -4636.21 & 32.83 \\
			$H_3$ & -42.32 & 41.57 & -132.25 & 3.78 & -120.84 & 4.55 & -62.49 & 7.77 \\
		    $H_4$ & -1810.46 & 27.57 & -2728.98 & 6.42 & -2236.28 & 4.55 & -1809.28 & 16.06 \\
		    $H_5$ & -291.54 & 22.62 & -508.79 & 3.39 & -330.32 & 4.95 & -291.54 & 9.07 \\
		    $H_6$ & -771.15 & 43.51 & -2471.77 & 4.04 & -1458.63 & 8.18 & -909.953 & 13.29 \\
		    $H_7$ & -23.81 & 35.34 & -59.67 &  3.15 & -43.96 & 4.13 & -31.63 &  6.04 \\
	    	$H_8$ & -42.82 & 44.17 & -221.83 & 3.38 & -184.75 & 4.12 & -49.26 & 17.04  \\
		    $H_9$ & -459.17 & 37.14 & -1134.22 & 3.25 & -881.56 & 5.16 & -493.75 & 10.87 \\
		    $H_{10}$ & -16.72 & 34.95 & -43.15 & 3.43 & -35.38 & 5.06 & -24.13 & 7.83 \\
		    $H_{11}$ & -14.78 & 40.41 & -45.79 & 13.74 & -20.29 & 27.12 & -10.99 & 306.13  \\
		    $H_{12}$ & -25.34 & 55.69 & -123.32 & 5.94 & -35.38 & 17.46 & -10.91  & 126.79 \\
		    $H_{13}$ & -295.29 & 59.77 & -1849.01 & 5.40 & -657.48 & 11.19 &-197.71 & 377.17 \\
		    $H_{14}$ & -48.99 & 62.93 & -212.21 & 5.34 & -124.71 & 9.93 & -26.64 & 51.71 \\
		    $H_{15}$ & -2666.89 & 48.12 & -5364.11 & 10.86 & -2818.64 & 28.59 & -2666.69 & 198.92 \\
		    $H_{16}$ & -362.99 & 42.74 & -1077.34 & 9.39 & -432.41 & 28.26 & -242.24 & 880.44 \\
		    $H_{17}$ & -1168.15 & 62.83 & -3254.47 & 5.50 & -2397.9  & 10.59 & -1056.01 & 390.77 \\
		    $H_{18}$ & -40.94 & 47.47 & -220.09 & 6.73 & -82.47 & 12.86 & -35.35 & 112.77 \\
		    $H_{19}$ & -100.11 & 60.13 & -380.15 & 4.17 & -124.99  & 14.98 & -78.97 & 228.44 \\
		    $H_{20}$ & -443.15 & 47.68 & -3433.21 & 4.47 & -1049.09  & 10.58 & -469.66 & 26.06 \\ \hline		
		\end{tabular}
	\end{center} 
\end{table}

\subsection{Constrained polynomial optimization problems} 

\label{subsec:res_con}

Now we present the numerical results for constrained polynomial optimization problems. First we minimize polynomials from set I used in \cref{subsec:res_uncon} over a basic semialgebraic set  $\mathbf{K} := \{(x_1, \ldots, x_n) \in \mathbb{R}^n \; | \; g_1 = 25 - (x_1^2 + \cdots + x_n^2) \geq 0\}$. Numerical results for the relaxation order $\hat{d} = 4$ with $\tau_0 = \tau_1$ are given in \cref{fig:con} and \cref{tab:num4} in \cref{sec:appendix}. These results give a qualitative similar picture to \cref{fig:uncon}. The computational gains for the TSSOS method are even more impressive, highlighting the potential of the rTSSOS method for constrained problems. We now focus on the question how the hyperparameters $\tau_0$ and $\tau_1$ should be chosen. The numerical results for the relaxation order $\hat{d} = 4$ with $(\tau_0, \tau_1) \in \set{(0.1 i, 0.1 j): i,j \in [9]}$ on polynomials from set I with $(n,s) = (8,21)$ are presented in \cref{fig:con_mesh}. According to these results it makes more sense to choose $\tau_0 \geq \tau_1$, since this yields relaxations returning tight bounds with larger computational savings compared to the case when $\tau_0 < \tau_1$.

\begin{figure}[tbhp]
	\centering
	\subfloat{\includegraphics[trim = 40 0 70 50, clip, width=0.5\linewidth]{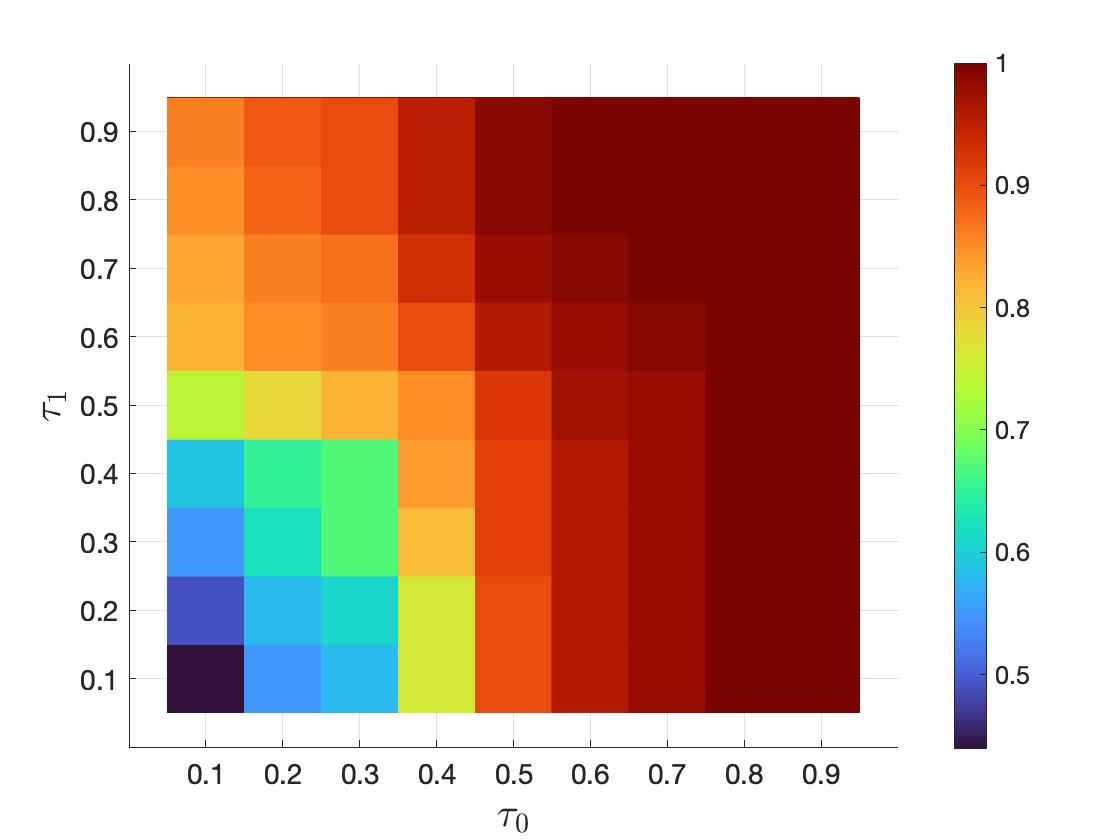}}
	\subfloat{\includegraphics[trim = 40 0 70 50, clip, width=0.5\linewidth]{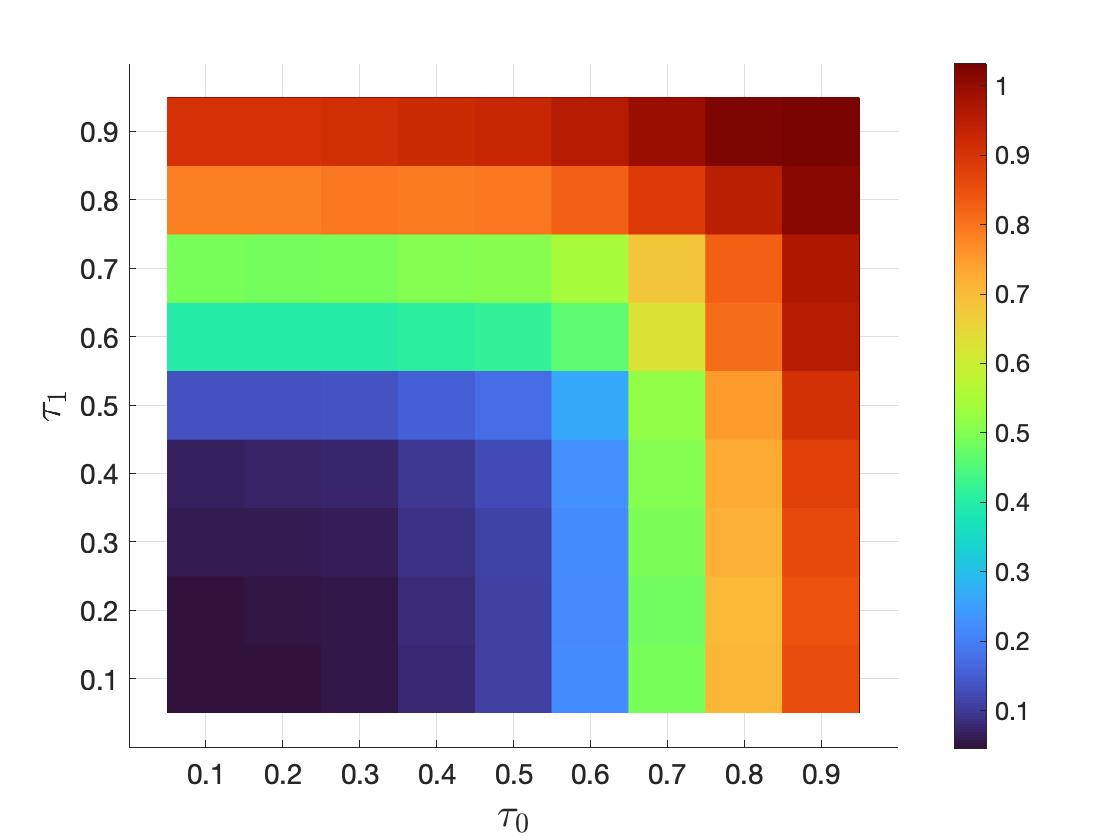}}
	\caption{Numerical results for the refined TSSOS on polynomials from set I with $(n,s) = (8,21)$, $\mathbf{K} := \{(x_1, \ldots, x_n) \in \mathbb{R}^n \; | \; g_1 = 25 - (x_1^2 + \cdots + x_n^2) \geq 0\}$ and the relaxation order $\hat{d} = 4$: the plot on the left shows the fraction of 100 problems solved to the accuracy $tol = 0.001$, i.e., with $\frac{|\Theta_{(\tau_0,\tau_1)} - \Theta_{t1}|}{|\Theta_{t1}|} \leq tol$, the plot on the right depicts the ratio of CPU time for the refined TSSOS to CPU time of TSSOS ($k=1$), i.e., $T_{(\tau_0,\tau_1)} / T_{t1}$, averaged over 100 problems, where $(\tau_0, \tau_1) \in \set{(0.1 i, 0.1 j): i,j \in [9]}$.}
	\label{fig:con_mesh}
\end{figure}

We now test the refined chordal-TSSOS method, i.e., the refined TSSOS approach implemented within the chordal-TSSOS method. For this we consider randomly generated polynomials defined by
\begin{equation*}
	f = \sum_{j = 1}^{s} c_j \mathbf{x}^{\mathbf{\alpha}_j} \in \text{\textbf{randpolyIII}}(n,2d,s)
\end{equation*}
constructed as follows: we randomly choose $s$ vectors $\mathbf{\alpha}_j$ in $\mathbb{N}_{2d}^{n}$ with random coefficients $c_j$ between $-1$ and $1$ and make sure that an obtained support contains at least one exponent $\alpha$ with $|\alpha| = 2d$. We generate 18 random polynomials $F_1, \ldots, F_{18}$:
\begin{align*}
	F_1,F_2,F_3 &\in \text{\textbf{randpolyIII}}(8,8,30) \\
	F_4,F_5,F_6 &\in \text{\textbf{randpolyIII}}(8,8,35) \\
	F_7,F_8,F_9 &\in \text{\textbf{randpolyIII}}(8,8,40) \\
	F_{10},F_{11},F_{12} &\in \text{\textbf{randpolyIII}}(10,8,30) \\
	F_{13},F_{14},F_{15} &\in \text{\textbf{randpolyIII}}(10,8,35) \\
	F_{16},F_{17},F_{18} &\in \text{\textbf{randpolyIII}}(10,8,40)
\end{align*}

The numerical results on polynomials $F_1, \ldots, F_{18}$ for the relaxation order $\hat{d} = 4$ and $\mathbf{K} := \{(x_1, \ldots, x_n) \in \mathbb{R}^n \; | \; g_1 = 9 - (x_1^2 + \cdots + x_n^2) \geq 0\}$ are listed in \cref{num9}, where the column $mc$ contains the maximal size of maximal cliques in the moment matrix $M_{\hat{d}}(\mathbf{y})$ and the localizing matrix $M_{\hat{d}-d_1}(g_1\mathbf{y})$, the obtained bound and the computing time are given in the columns $Opt$ and $Time$, respectively. On these polynomials in comparison with the first iterative step of the chordal-TSSOS the refined version of this method allows to reduce the computing time without a large loss in the quality of bounds.

\begin{table}[tbhp]
	\footnotesize
	\captionsetup{position=top} %<- Needed for using subtables created with the subfig package
	\caption{Numerical results on polynomials $F_i$, $i = 1, \ldots, 18$ for the constrained case with $\mathbf{K} := \{(x_1, \ldots, x_n) \in \mathbb{R}^n \; | \; g_1 = 9 - (x_1^2 + \cdots + x_n^2) \geq 0\}$ and the relaxation order $\hat{d} = 4$: on these polynomials the refined chordal-TSSOS allows to reduce the computing time without a large loss in the quality of bounds compared to the chordal-TSSOS ($k = 1$)}\label{num9} 
	\begin{center}
		\begin{tabular}{|c|c|c|c|c|c|c|c|c|c|}  \hline
			 & \multicolumn{3}{|c|}{\multirow{2}{*}{chordal-TSSOS} } & \multicolumn{6}{|c|}{refined chordal-TSSOS} \\  \cline{5-10}
			 & \multicolumn{3}{|c|}{}&  \multicolumn{3}{|c|}{$\tau = (0.1,0.1)$} &  \multicolumn{3}{|c|}{$\tau = (0.3,0.3)$} \\ \cline{2-10}
			 & $mc$ & $Opt$ & $Time$ & $mc$ & $Opt$ & $Time$ & $mc$ & $Opt$ & $Time$   \\ \hline
			 $F_1$ & (54,41) & -120.14 & 25.73 & (45,10) & -120.19 & 3.21 & (45,15) & -120.80 & 5.23 \\
			 $F_2$ & (57,38) & -344.14 & 25.98 & (45,9) & -344.15 & 2.95 & (45,17) & -344.14 & 6.10 \\
			 $F_3$ & (60,42) & -103.49 & 28.06 & (45,9) & -103.70 & 3.42 & (45,14) & -103.54 & 4.40 \\
			 $F_4$ & (74,44) & -1360.37 & 32.10 & (45,13) & -1360.37 & 4.43 & (70,47) & -1360.37 &24.60 \\
			 $F_5$ & (69,41) & -328.72 & 30.55 & (45,10) & -328.72 & 3.60 & (45,31) & -328.72 & 11.32 \\
			 $F_6$ & (76,46) & -137.16 & 34.24 & (45,21) & -137.16 & 7.52 & (45,42) & -137.16 & 22.37 \\
			 $F_7$ & (70,47) & -579.07 & 27.83 & (45,14) & -579.19 & 4.31 & (45,28) & -579.07 & 8.69 \\
			 $F_8$ & (72,44) & -422.31 & 45.24 & (45,23) & -422.44 & 8.54 & (45,26) & -422.31 & 10.71 \\
			 $F_9$ & (98,48) & -444.38 & 51.91 & (45,20) & -444.38 & 5.60 & (45,46) & -444.38 & 39.65 \\
			 $F_{10}$ & (66,42) & -302.29 & 50.12 & (66,11) & -302.29 & 5.46 & (66,11) & -302.29 & 5.74 \\
			 $F_{11}$ & (66,28) & -106.61 & 20.48 & (66,11) & -106.61 & 4.63 & (66,11) & -106.61 &  5.18 \\
			 $F_{12}$ & (66,39) & -97.77 & 36.22 & (66,11) & -158.01 & 4.24 & (66,11) & -97.77 & 4.57 \\
			 $F_{13}$ & (74,48) & -103.82 & 141.96 & (66,11) & -105.65 & 5.29 & (66,11) &-105.65 & 6.96 \\
			 $F_{14}$ & (66,45) & -196.48 & 53.68 & (66,11)  & -196.54 & 4.92 & (66,11)  & -196.48 & 5.36 \\
			 $F_{15}$ & (66,47) & -38.88 & 89.59 & (66,11) & -38.97 & 5.01 & (66,11) & -38.91 & 6.58 \\
			 $F_{16}$ & (89,57) & -238.29 & 147.44 & (66,16) & -238.29 & 8.37 & (66,35) & -238.29 & 31.12 \\
			 $F_{17}$ & (84,55) & -166.62 & 148.59 & (66,11) & -166.80 & 5.64 & (66,12) & -166.80 & 6.99 \\
			 $F_{18}$ & (66,43) & -116.91 & 56.61 & (66,24) & -116.91 & 10.92 & (66,31) & -116.91 & 21.22 \\ \hline
		\end{tabular}
	\end{center} 
\end{table}

\section{Conclusion and Outlook}

\label{sec:conclusion}

We have provided a new approach that refines TSSOS iterations using combinatorial optimization and results in block-diagonal matrices with reduced maximum block sizes. Numerical results on a benchmark library show the large potential for computational speedup for unconstrained and constrained polynomial optimization, while obtaining almost identical bounds in comparison to established methods.

One direction of further research is to investigate other strategies for generating a partition $I^{(\tau)}$ with $\tau = k - 1 + \varepsilon$, $\varepsilon \in (0,1)$. For instance, $I^{(\tau)}$ can also be obtained by solving an IP problem that minimizes the width of $I^{(\tau)}$  and restricts the corresponding binary matrix $B_{\mathscr{A}}^{(\tau)}$ to satisfy the requirements:
\begin{enumerate}
	\item for all vectors $\nu \in \left(\mathscr{S}^{(k-1)}\right)_2 \setminus \set{\mathbf{0}}$, $\supp(B_{\mathscr{A}}^{(\tau)})$ should contain at least $\varepsilon \sum_{(\delta,\sigma) \in J_{\nu}} |\mathscr{B}_{\delta}| \cdot |\mathscr{B}_{\sigma}|$ elements from $$\set{(\beta,\gamma) \in \supp(B_{\mathscr{A}}^{(k)}) | (\beta + \gamma)_2 = \nu},$$ where $J_{\nu} := \set{(\delta,\sigma) \in \supp(\overline{C^{(k)}}) | \delta + \sigma = \nu}$ with $C^{(k)}$ defined in \cref{eq:C},
	\item $\set{(\beta,\gamma) \in \supp(B_{\mathscr{A}}^{(k)}) : (\beta + \gamma)_2 = \mathbf{0} } \subseteq \supp(B_{\mathscr{A}}^{(\tau)})$
\end{enumerate}
Utilization of this or some other alternative strategy might potentially improve the approach. Another question left for further investigation is how to choose an optimal parameter value $\tau$ producing tight bounds and keeping the computational costs small at the same time.

\clearpage

\appendix
\section{Numerical results}

\label{sec:appendix}

\begin{table}[tbhp]
	\scriptsize
	\captionsetup{position=top} %<- Needed for using subtables created with the subfig package
	\caption{IP problem data for sets I, II and III: columns "max(\#var)" and "max(\#con)" contain the maximal number of variables and constaints over all IP problems (for all $\tau$ and all polynomials), columns "mean(\#var)" and "mean(\#con)" display the corresponding average values, "\#prob" is an average amount of IP problems solved to obtain an rTSSOS relaxation and an average time to get a relaxation is given in collumn "time". }\label{tab:ip_sets} 
	\begin{center}
		\begin{tabular}{|c|c|c|c|c|c|c|c|} \hline
			Set & $(n,d,s)$ & max(\#var) & mean(\#var) & max(\#con) & mean(\#con) & \#prob& time \\ \hline
			\multirow{10}{*}{I}& $(8,4,17)$  & 140 & 11.42 & 1483 & 26.84 & 7.40 & 0.276 \\
			& $(8,4,19)$ & 255 & 12.43 & 5339 & 36.95 & 9.31 & 0.356 \\
			& $(8,4,21)$ & 215 & 13.71 & 4020 & 49.83 & 11.21	& 0.447 \\
			& $(9,4,19)$ & 121 & 12.47 & 1697 & 28.96	& 8.46 & 0.320 \\
			& $(9,4,21)$ & 162 & 12.94 & 2483 & 34.48 & 10.54 & 0.409 \\
			& $(9,4,23)$ & 161 & 13.94 & 2481 & 41.04 & 12.33 & 0.487 \\
			& $(10,4,21)$ & 94 & 12.50 & 1111 & 26.08 &	9.65 & 0.379 \\
			& $(10,4,23)$ & 203 & 13.61 & 3473 & 33.29 & 11.57 & 0.465 \\
			& $(11,4,23)$ & 68 & 13.99 & 419 & 28.58 & 10.74 & 0.444 \\
			& $(11,4,25)$ & 197 & 14.34 & 3453 & 32.14 & 12.69 & 0.535 \\ \hline
			\multirow{10}{*}{II}& $(8,4,17)$ &  147 & 22.24 & 1697 & 59.12 & 7.73 & 0.310 \\
			& $(8,4,19)$ & 302 & 26.72 & 6928 & 120.82 & 9.81 & 0.458\\
			& $(8,4,21)$ & 412 & 29.67 & 11002 & 165.23 & 11.62 & 0.638 \\
			& $(9,4,19)$ & 196 & 21.72 & 3451 & 53.13 & 8.89 & 0.353 \\
			& $(9,4,21)$ &  439 & 24.69 & 12217 & 84.91 & 10.89 & 0.479 \\
			& $(9,4,23)$ & 360 & 26.87 & 8819 & 115.42 & 12.79 & 0.584 \\
			& $(10,4,21)$ & 121 & 20.64 & 1697 & 43.79 & 9.91 & 0.388\\
			& $(10,4,23)$ & 326 & 22.53 & 7827 & 58.76 & 11.90 & 0.478 \\
			& $(11,4,23)$ & 122 & 20.38 & 1098 & 41.27 & 10.95 & 0.446 \\
			& $(11,4,25)$ & 177 & 21.34 & 2937 & 46.99 & 12.92 & 0.531\\ \hline
			\multirow{10}{*}{III}& $(8,4,17)$ & 11 & 5.23 & 36 & 10.04 & 6.75 & 0.239 \\
			& $(8,4,19)$ & 11 & 5.38 & 36 & 10.59 & 8.69 & 0.313 \\
			& $(8,4,21)$ & 14 & 5.61 & 43 & 11.51 & 9.78 & 0.355 \\
			& $(9,4,19)$ & 11 & 5.36 & 36 & 10.30 & 7.82 & 0.290 \\
			& $(9,4,21)$ & 11 & 5.36 & 36 & 10.43 & 9.27 & 0.344 \\
			& $(9,4,23)$ & 13 & 5.51 & 41 & 11.02 & 11.01 & 0.407 \\
			& $(10,4,21)$ & 11 & 5.27 & 36 & 10.04 & 8.72 & 0.336 \\
			& $(10,4,23)$ & 12 & 5.51 & 36 & 10.69  & 9.94 & 0.382 \\
			& $(11,4,23)$ & 11 & 5.39 & 36 & 10.29 & 9.40 & 0.391 \\
			& $(11,4,25)$ &11 & 5.48 & 36 & 10.71 & 11.11 & 0.467 \\ \hline
		\end{tabular}
	\end{center} 
\end{table}

	\begin{table}[tbhp]
	\scriptsize
	\captionsetup{position=top} %<- Needed for using subtables created with the subfig package
	\caption{Numerical results on set I: the amount of problems solved to the accuracy $tol$ and the average computing time for the chordal ($c1$ and $c2$), block ($t1$) and refined TSSOS for $\tau \in \{0.4,0.45,0.5,0.55,0.6,0.65,0.7\}$.}\label{tab:num1} 
	\begin{center}
		\begin{tabular}{|c|c|c|c|c|c||c|c|c|c|c|c|} \hline
			\multicolumn{6}{|c||}{$n = 8, d = 4, s = 17$} & 	\multicolumn{6}{|c|}{$n = 8, d = 4, s = 19$} \\ \hline
			\multicolumn{1}{|c|}{\multirow{2}{*}{$M$}} & \multicolumn{4}{c|}{$tol$} & \multicolumn{1}{c||}{\multirow{2}{*}{time}} & \multicolumn{1}{|c|}{\multirow{2}{*}{$M$}} & \multicolumn{4}{c|}{$tol$} & \multicolumn{1}{c|}{\multirow{2}{*}{time}} \\ \cline{2-5} \cline{8-11}
			& 0.001 & 0.01 & 0.1 & 0.25  & & & 0.001 & 0.01 & 0.1 & 0.25  &  \\ \hline
			$c1$ & 89 & 94 & 97 & 97 & 0.422 & $c1$ & 58 & 64 & 75 & 91 & 0.424\\ 
			$c2$ & 99 & 99 & 100 & 100 & 4.443 & $c2$  & 97 & 97 & 100 & 100 & 16.290\\ 
			$t1$ & 100 & 100 & 100 & 100 & 14.825 & $t1$ & 100 & 100 & 100 & 100 & 38.648\\ 
			$0.4$ & 95 & 95 & 98 & 100 & 1.356 & $0.4$ & 68 & 71 & 81 & 91 & 2.293 \\ 
			$0.45$ & 95 & 95 & 98 & 100 & 1.386 & $0.45$ & 73 & 76 & 81 & 92 & 2.678 \\ 
			$0.5$ & 96 & 96 & 99 & 100 & 1.520 & $0.5$ & 80 & 81 & 83 & 93 & 2.917 \\ 
			$0.55$ & 98 & 98 & 99 & 100 & 2.275  & $0.55$ & 86 & 87 & 90 & 97 & 5.119\\ 
			$0.6$ & 98 & 98 & 99 & 100 & 2.432 & $0.6$ & 87 & 89 & 92 & 100 & 5.419 \\ 
			$0.65$ & 100 & 100 & 100 & 100 & 3.611 & $0.65$ & 96 & 97 & 98 & 100 & 11.835 \\ 
			$0.7$ & 100 & 100 & 100 & 100 & 5.816 & $0.7$ & 97 & 97 & 97 & 100 & 17.239 \\ \hline \hline
			\multicolumn{6}{|c||}{$n = 8, d = 4, s = 21$} & 	\multicolumn{6}{|c|}{$n = 9, d = 4, s = 19$} \\ \hline
			\multicolumn{1}{|c|}{\multirow{2}{*}{$M$}} & \multicolumn{4}{c|}{$tol$} & \multicolumn{1}{c||}{\multirow{2}{*}{time}} & \multicolumn{1}{|c|}{\multirow{2}{*}{$M$}} & \multicolumn{4}{c|}{$tol$} & \multicolumn{1}{c|}{\multirow{2}{*}{time}} \\ \cline{2-5} \cline{8-11}
			& 0.001 & 0.01 & 0.1 & 0.25  & & & 0.001 & 0.01 & 0.1 & 0.25  &  \\ \hline
			$c1$ & 52& 58 & 81 & 89 & 0.519 & $c1$ & 88 & 91 & 97 & 98 & 0.721\\ 
			$c2$ & 98 & 99 & 100 & 100 & 41.089 & $c2$  & 97 & 98 & 99 & 99 & 7.249 \\ 
			$t1$ & 100 & 100 & 100 & 100 & 50.649 & $t1$ & 100 & 100 & 100 & 100 & 36.777 \\ 
			$0.4$ & 71 & 77 & 86 & 95 & 4.380 & $0.4$ & 97 & 97 & 98 & 98 & 2.365 \\ 
			$0.45$ & 78 & 82 & 89 & 95 & 5.117 & $0.45$ & 97 & 97 & 98 & 98 & 2.407 \\ 
			$0.5$ & 86 & 90 & 92 & 97 & 6.623 & $0.5$ & 97 & 97 & 98 & 98 & 2.712 \\ 
			$0.55$ & 94 & 96 & 98 & 99 & 14.820 & $0.55$ & 98 & 99 & 100 & 100 & 4.286 \\ 
			$0.6$ & 95 & 97 & 98 & 100 & 16.061 & $0.6$ & 99 & 99 & 100 & 100 & 4.207 \\ 
			$0.65$ & 96 & 98 & 99 & 100 & 32.672 & $0.65$ & 99 & 99 & 100 & 100 & 6.891 \\
			$0.7$ & 98 & 100 & 100 & 100 & 39.710 & $0.7$ & 99 & 99 & 100 & 100 & 10.146 \\ \hline \hline
			\multicolumn{6}{|c||}{$n = 9, d = 4, s = 21$} & 	\multicolumn{6}{|c|}{$n = 9, d = 4, s = 23$} \\ \hline
			\multicolumn{1}{|c|}{\multirow{2}{*}{$M$}} & \multicolumn{4}{c|}{$tol$} & \multicolumn{1}{c||}{\multirow{2}{*}{time}} & \multicolumn{1}{|c|}{\multirow{2}{*}{$M$}} & \multicolumn{4}{c|}{$tol$} & \multicolumn{1}{c|}{\multirow{2}{*}{time}} \\ \cline{2-5} \cline{8-11}
			& 0.001 & 0.01 & 0.1 & 0.25  & & & 0.001 & 0.01 & 0.1 & 0.25  &  \\ \hline
			$c1$ & 58 & 67 & 89 & 97 & 0.811 & $c1$ & 29 & 33 & 62 & 78 & 1.299 \\ 
			$c2$ & 98 & 98 & 99 & 99 & 43.349 & $c2$ & 98 & 99 & 100 & 100 & 188.425 \\ 
			$t1$ & 100 & 100 & 100 & 100 & 126.055 & $t1$ & 100 & 100 & 100 & 100 & 192.679\\ 
			$0.4$ & 74 & 77 & 88 & 92 & 3.034 & $0.4$ & 59 & 63 & 77 & 87 & 4.635 \\ 
			$0.45$ & 75 & 78 & 89 & 93 & 3.181 & $0.45$ & 63 & 71 & 81 & 89 & 5.314 \\ 
			$0.5$ & 78 & 82 & 92 & 93 & 3.430 & $0.5$ & 69 & 76 & 84 & 92 & 6.562 \\ 
			$0.55$ & 87 & 89 & 94 & 98 & 3.942 & $0.55$ & 85 & 86 & 95 & 99 & 13.608 \\ 
			$0.6$ & 86 & 89 & 94 & 97 & 5.727 & $0.6$ & 85 & 87 & 95 & 98 & 15.283 \\ 
			$0.65$ & 94 & 96 & 99 & 99 & 12.122 & $0.65$ & 100 & 100 & 100 & 100 & 40.967 \\ 
			$0.7$ & 96 & 97 & 99 & 99 & 17.512 & $0.7$ & 99 & 100 & 100 & 100 & 53.307 \\ \hline  \hline	
			\multicolumn{6}{|c||}{$n = 10, d = 4, s = 21$} & 	\multicolumn{6}{|c|}{$n = 10, d = 4, s = 23$} \\ \hline
			\multicolumn{1}{|c|}{\multirow{2}{*}{$M$}} & \multicolumn{4}{c|}{$tol$} & \multicolumn{1}{c||}{\multirow{2}{*}{time}} & \multicolumn{1}{|c|}{\multirow{2}{*}{$M$}} & \multicolumn{4}{c|}{$tol$} & \multicolumn{1}{c|}{\multirow{2}{*}{time}} \\ \cline{2-5} \cline{8-11}
			& 0.001 & 0.01 & 0.1 & 0.25  & & & 0.001 & 0.01 & 0.1 & 0.25  &  \\ \hline
			$c1$ & 72 & 76 & 87 & 95 & 1.435 & $c1$ & 49 & 53 & 74 & 89 & 1.889 \\ 
			$c2$ & 98 & 99 & 100 & 100 & 18.108 & $c2$ & 97 & 97 & 99 & 99 & 119.080 \\ 
			$t1$ & 100 & 100 & 100 & 100 & 101.374 & $t1$ & 100 & 100 & 100 & 100 & 327.538\\ 
			$0.4$ & 84 & 86 & 93 & 97 & 2.914 & $0.4$ & 57 & 61 & 74 & 84 & 4.664 \\ 
			$0.45$ & 84 & 86 & 94 & 98 & 3.145 & $0.45$ & 59 & 63 & 76 & 83 & 5.192 \\ 
			$0.5$ & 85 & 86 & 93 & 98 & 3.363 & $0.5$ & 62 & 67 & 80 & 85 & 5.424 \\ 
			$0.55$ & 87 & 89 & 94 & 98 & 3.942  & $0.55$ & 71 & 75 & 87 & 92 & 9.887 \\ 
			$0.6$ & 89 & 90 & 95 & 98 & 4.126 & $0.6$ & 76 & 79 & 90 & 93 & 10.908 \\ 
			$0.65$ & 93 & 94 & 98 & 98 & 6.484& $0.65$ & 81 & 85 & 95 & 96 & 22.169 \\
			$0.7$ & 94 & 95 & 98 & 99 & 9.678 & $0.7$ & 90 & 93 & 98 & 98 & 30.888 \\ \hline \hline
			\multicolumn{6}{|c||}{$n = 11, d = 4, s = 23$} & 	\multicolumn{6}{|c|}{$n = 11, d = 4, s = 25$} \\ \hline
			\multicolumn{1}{|c|}{\multirow{2}{*}{$M$}} & \multicolumn{4}{c|}{$tol$} & \multicolumn{1}{c||}{\multirow{2}{*}{time}} & \multicolumn{1}{|c|}{\multirow{2}{*}{$M$}} & \multicolumn{4}{c|}{$tol$} & \multicolumn{1}{c|}{\multirow{2}{*}{time}} \\ \cline{2-5} \cline{8-11}
			& 0.001 & 0.01 & 0.1 & 0.25  & & & 0.001 & 0.01 & 0.1 & 0.25  &  \\ \hline
			$c1$ & 70 & 71 & 84 & 93 & 2.312 & $c1$ & 39 & 42 & 66 & 82 & 2.469\\ 
			$c2$ & 97 & 98 & 99 & 100 & 21.751 & $c2$ & 91 & 91 & 97 & 99 & 172.247 \\ 
			$t1$ & 100 & 100 & 100 & 100 & 144.074 & $t1$ & 100 & 100 & 100 & 100 & 569.459 \\ 
			$0.4$ & 81 & 83 & 92 & 96 & 4.839 & $0.4$ & 49 & 53 & 73 & 88 & 6.389 \\ 
			$0.45$ & 82 & 83 & 92 & 96 & 4.894 & $0.45$ & 53 & 57 & 77 & 89 & 6.820 \\ 
			$0.5$ & 82 & 83 & 92 & 96 & 5.055 & $0.5$ & 55 & 58 & 78 & 91 & 7.172 \\ 
			$0.55$ & 84 & 86 & 95 & 97 & 6.268 & $0.55$ & 62 & 67 & 82 & 92 & 9.817 \\ 
			$0.6$ & 83 & 85 & 94 & 97 & 6.372 & $0.6$ & 64 & 70 & 85 & 92 & 9.726 \\ 
			$0.65$ & 89 & 90 & 98 & 99 & 10.408 & $0.65$ & 74 & 77 & 86 & 93 & 19.399 \\
			$0.7$ & 92 & 94 & 98 & 99 & 13.447 & $0.7$ & 78 & 84 & 91 & 96 & 33.441 \\ \hline
		\end{tabular}
	\end{center} 
\end{table}

	\begin{table}[tbhp]
	\scriptsize
	\captionsetup{position=top} %<- Needed for using subtables created with the subfig package
	\caption{Numerical results on set II: the amount of problems solved to the accuracy $tol$ and the average computing time for the chordal ($c1$ and $c2$), block ($t1$) and refined TSSOS for $\tau \in \{0.4,0.45,0.5,0.55,0.6,0.65,0.7\}$, the amount of problems not solved within the time limit is given in round brackets in column $M$. }\label{tab:num5} 
	\begin{center}
		\begin{tabular}{|c|c|c|c|c|c||c|c|c|c|c|c|} \hline
			\multicolumn{6}{|c||}{$n = 8, d = 4, s = 17$} & 	\multicolumn{6}{|c|}{$n = 8, d = 4, s = 19$} \\ \hline
			\multicolumn{1}{|c|}{\multirow{2}{*}{$M$}} & \multicolumn{4}{c|}{$tol$} & \multicolumn{1}{c||}{\multirow{2}{*}{time}} & \multicolumn{1}{|c|}{\multirow{2}{*}{$M$}} & \multicolumn{4}{c|}{$tol$} & \multicolumn{1}{c|}{\multirow{2}{*}{time}} \\ \cline{2-5} \cline{8-11}
			& 0.001 & 0.01 & 0.1 & 0.25  & & & 0.001 & 0.01 & 0.1 & 0.25  &  \\ \hline
			$c1$ & 90 & 92 & 97 & 98 & 0.515 & $c1$ & 60 & 65 & 81 & 90 & 0.709\\ 
			$c2$ & 100 & 100 & 100 & 100 & 20.403 & $c2$  & 99 & 100 & 100 & 100 & 82.667\\ 
			$t1$ & 100 & 100 & 100 & 100 & 22.099 & $t1$ & 100 & 100 & 100 & 100 & 53.532\\ 
			$0.4$ & 97 & 97 & 99 & 100 & 1.131  & $0.4$ & 97 & 97 & 99 & 100 & 2.009 \\
			$0.45$ & 99 & 99 & 100 & 100 & 1.455 & $0.45$ & 100 & 100 & 100 & 100 & 3.292 \\ 
			$0.5$ & 99 & 99 & 100 & 100 & 1.453 & $0.5$ & 100 & 100 & 100 & 100 & 3.543 \\ 
			$0.55$ & 100 & 100 & 100 & 100 & 1.597  & $0.55$ & 100 & 100 & 100 & 100 & 6.391 \\ 
			$0.6$ & 100 & 100 & 100 & 100 & 2.058 & $0.6$ & 100 & 100 & 100 & 100 & 9.509 \\ 
			$0.65$ & 100 & 100 & 100 & 100 & 2.938 & $0.65$ & 100 & 100 & 100 & 100 & 18.293 \\  
			$0.7$ & 100 & 100 & 100 & 100 & 3.924 & $0.7$ & 100 & 100 & 100 & 100 & 28.847 \\ \hline \hline	
			\multicolumn{6}{|c||}{$n = 8, d = 4, s = 21$} & 	\multicolumn{6}{|c|}{$n = 9, d = 4, s = 19$} \\ \hline
			\multicolumn{1}{|c|}{\multirow{2}{*}{$M$}} & \multicolumn{4}{c|}{$tol$} & \multicolumn{1}{c||}{\multirow{2}{*}{time}} & \multicolumn{1}{|c|}{\multirow{2}{*}{$M$}} & \multicolumn{4}{c|}{$tol$} & \multicolumn{1}{c|}{\multirow{2}{*}{time}} \\ \cline{2-5} \cline{8-11}
			& 0.001 & 0.01 & 0.1 & 0.25  & & & 0.001 & 0.01 & 0.1 & 0.25  &  \\ \hline
			$c1$ & 61 & 70 & 85 & 89 & 0.978 & $c1$ & 84 & 85 & 89 & 91 & 1.078\\ 
			$c2$ & 100 & 100 & 100 & 100 & 155.100 & $c2$  & 98 & 98 & 100 & 100 & 68.172 \\ 
			$t1$ & 100 & 100 & 100 & 100 & 59.671 & $t1$ & 100 & 100 & 100 & 100 & 72.764 \\ 
			$0.4$ & 99 & 99 & 99 & 99 & 4.152 & $0.4$ & 89 & 91 & 93 & 96 & 1.680 \\ 
			$0.45$ & 100 & 100 & 100 & 100 & 8.234 & $0.45$ & 92 & 94 & 95 & 100 & 1.805 \\ 
			$0.5$ & 100 & 100 & 100 & 100 & 11.753 & $0.5$ & 92 & 93 & 95 & 99 & 1.832 \\ 
			$0.55$ & 100 & 100 & 100 & 100 & 25.417 & $0.55$ & 95 & 96 & 97 & 99 & 1.909 \\ 
			$0.6$ & 100 & 100 & 100 & 100 & 34.430 & $0.6$ &  95 & 96 & 97 & 98 & 2.091 \\ 
			$0.65$ & 100 & 100 & 100 & 100 & 51.482 & $0.65$ & 96 & 97 & 99 & 99 & 2.479 \\
			$0.7$ & 100 & 100 & 100 & 100 & 59.696 & $0.7$ & 96 & 97 & 98 & 99 & 3.213 \\ \hline \hline	
			\multicolumn{6}{|c||}{$n = 9, d = 4, s = 21$} & 	\multicolumn{6}{|c|}{$n = 9, d = 4, s = 23$} \\ \hline
			\multicolumn{1}{|c|}{\multirow{2}{*}{$M$}} & \multicolumn{4}{c|}{$tol$} & \multicolumn{1}{c||}{\multirow{2}{*}{time}} & \multicolumn{1}{|c|}{\multirow{2}{*}{$M$}} & \multicolumn{4}{c|}{$tol$} & \multicolumn{1}{c|}{\multirow{2}{*}{time}} \\ \cline{2-5} \cline{8-11}
			& 0.001 & 0.01 & 0.1 & 0.25  & & & 0.001 & 0.01 & 0.1 & 0.25  &  \\ \hline
			$c1$ & 60 & 64 & 77 & 86 & 0.956 & $c1$ & 36 & 42 & 58 & 72 & 1.147 \\ 
			$c2$ & 97 & 98 & 100 & 100 & 593.791 & $c2$  & 100 & 100 & 100 & 100 & 1452.443 \\ 
			$t1$ & 99 & 100 & 100 & 100 & 199.432 & $t1$ & 100 & 100 & 100 & 100 & 256.519 \\ 
			$0.4$ & 74 & 76 & 90 & 96 & 2.568 & $0.4$ & 77 & 80 & 91 & 95 & 4.049 \\ 
			$0.45$ & 85 & 86 & 94 & 98 & 2.974& $0.45$ & 90 & 92 & 96 & 99 & 5.746  \\ 
			$0.5$ & 88 & 90 & 95 & 99 & 3.092 & $0.5$ & 90 & 93 & 98 & 99 & 6.801 \\ 
			$0.55$ & 92 & 94 & 99 & 100 & 4.123 & $0.55$ & 96 & 98 & 99 & 99 & 11.639 \\ 
			$0.6$ & 94 & 96 & 99 & 100 & 5.198 & $0.6$ & 100 & 100 & 100 & 100 & 19.413 \\ 
			$0.65$ & 99 & 100 & 100 & 100 & 11.056 & $0.65$ & 100 & 100 & 100 & 100 & 47.260 \\
			$0.7$ & 98 & 100 & 100 & 100 & 23.030 & $0.7$ & 100 & 100 & 100 & 100 & 97.0419 \\ \hline  \hline	
			\multicolumn{6}{|c||}{$n = 10, d = 4, s = 21$} & 	\multicolumn{6}{|c|}{$n = 10, d = 4, s = 23$} \\ \hline
			\multicolumn{1}{|c|}{\multirow{2}{*}{$M$}} & \multicolumn{4}{c|}{$tol$} & \multicolumn{1}{c||}{\multirow{2}{*}{time}} & \multicolumn{1}{|c|}{\multirow{2}{*}{$M$}} & \multicolumn{4}{c|}{$tol$} & \multicolumn{1}{c|}{\multirow{2}{*}{time}} \\ \cline{2-5} \cline{8-11}	
			& 0.001 & 0.01 & 0.1 & 0.25  & & & 0.001 & 0.01 & 0.1 & 0.25  &  \\ \hline
			$c1$ & 84 & 86 & 92 & 96 & 1.449 & $c1$ & 48 & 50 & 67 & 80 & 1.491 \\ 
			$c2$ & 98 & 99 & 100 & 100 & 96.794 & $c2$(6) & 93 & 94 & 94 & 94 & 1325.003\\ 
			$t1$ & 100 & 100 & 100 & 100 & 128.946 & $t1$ & 100 & 100 & 100 & 100 & 657.311 \\ 
			$0.4$ & 94 & 95 & 97 & 98 & 2.845 & $0.4$ & 72 & 74 & 85 & 91 & 3.768 \\ 
			$0.45$ & 94 & 95 & 97 & 99 & 3.061 & $0.45$ & 72 & 74 & 89 & 94 & 4.077  \\ 
			$0.5$ & 96 & 97 & 99 & 99 & 3.082 & $0.5$ & 74 & 76 & 89 & 94 & 4.206 \\ 
			$0.55$ & 95 & 96 & 99 & 100 & 3.281 & $0.55$ & 77 & 81 & 89 & 97 & 4.485 \\ 
			$0.6$ & 95 & 96 & 98 & 99 & 3.536 & $0.6$ & 82 & 83 & 91 & 98 & 4.923 \\ 
			$0.65$ & 96 & 97 & 99 & 99 & 4.079 & $0.65$ & 87 & 87 & 95 & 99 & 6.436 \\
			$0.7$ &  95 & 97 & 99 & 99 & 5.594 & $0.7$ & 92 & 92 & 99 & 100 & 10.289 \\ \hline \hline
			\multicolumn{6}{|c||}{$n = 11, d = 4, s = 23$} & 	\multicolumn{6}{|c|}{$n = 11, d = 4, s = 25$} \\ \hline
			\multicolumn{1}{|c|}{\multirow{2}{*}{$M$}} & \multicolumn{4}{c|}{$tol$} & \multicolumn{1}{c||}{\multirow{2}{*}{time}} & \multicolumn{1}{|c|}{\multirow{2}{*}{$M$}} & \multicolumn{4}{c|}{$tol$} & \multicolumn{1}{c|}{\multirow{2}{*}{time}} \\ \cline{2-5} \cline{8-11}
			& 0.001 & 0.01 & 0.1 & 0.25  & & & 0.001 & 0.01 & 0.1 & 0.25  &  \\ \hline
			$c1$ & 74 & 75 & 87 & 91 & 2.801 & $c1$ & 35 & 40 & 53 & 74 & 2.415\\ 
			$c2$ & 94 & 95 & 96 & 96 & 98.126 & $c2$(6) & 88 & 90 & 93 & 94 & 1211.669\\ 
			$t1$ & 99 & 99 & 100 & 100 & 140.913 & $t1$ & 100 & 100 & 100 & 100  & 888.690 \\ 
			$0.4$ & 86 & 86 & 93 & 99 & 4.174 & $0.4$ & 57 & 61 & 74 & 86 & 5.459 \\ 
			$0.45$ & 85 & 85 & 93 & 99 & 4.377 & $0.45$ & 57 & 63 & 80 & 90 & 5.773 \\ 
			$0.5$ & 86 & 86 & 94 & 99 & 4.414 & $0.5$ & 58 & 63 & 80 & 90 & 5.799 \\ 
			$0.55$ & 85 & 86 & 94 & 99 & 4.542 & $0.55$ & 60 & 65 & 80 & 92 & 6.213 \\ 
			$0.6$ & 87 & 88 & 95 & 99 & 4.788 & $0.6$ & 63 & 69 & 81 & 91 & 6.551 \\ 
			$0.65$ & 89 & 90 & 95 & 98 & 5.028 & $0.65$ & 69 & 74 & 84 & 92 & 7.696 \\
			$0.7$ & 91 & 91 & 95 & 98 & 5.391 & $0.7$ & 71 & 74 & 85 & 93 & 8.858 \\ \hline
		\end{tabular}
	\end{center} 
\end{table}

\begin{table}[tbhp]
	\scriptsize
	\captionsetup{position=top} %<- Needed for using subtables created with the subfig package
	\caption{Numerical results on set III: the amount of problems solved to the accuracy $tol$ and the average computing time for the chordal ($c1$ and $c2$), block ($t1$) and refined TSSOS for $\tau \in \{0.4,0.45,0.5,0.55,0.6,0.65,0.7\}$.}\label{tab:num7} 
	\begin{center}
		\begin{tabular}{|c|c|c|c|c|c||c|c|c|c|c|c|} \hline
			\multicolumn{6}{|c||}{$n = 8, d = 4, s = 17$} & 	\multicolumn{6}{|c|}{$n = 8, d = 4, s = 19$} \\ \hline
			\multicolumn{1}{|c|}{\multirow{2}{*}{$M$}} & \multicolumn{4}{c|}{$tol$} & \multicolumn{1}{c||}{\multirow{2}{*}{time}} & \multicolumn{1}{|c|}{\multirow{2}{*}{$M$}} & \multicolumn{4}{c|}{$tol$} & \multicolumn{1}{c|}{\multirow{2}{*}{time}} \\ \cline{2-5} \cline{8-11}
			& 0.001 & 0.01 & 0.1 & 0.25  & & & 0.001 & 0.01 & 0.1 & 0.25  &  \\ \hline
			$c1$ & 85 & 92 & 99 & 100 & 0.400 & $c1$ & 70 & 80 & 92 & 94 & 0.384\\ 
			$c2$ & 99 & 100 & 100 & 100 & 1.217 & $c2$  & 97 & 97 & 99 & 99 & 2.908 \\ 
			$t1$ & 100 & 100 & 100 & 100 & 9.674 & $t1$ & 100 & 100 & 100 & 100 & 21.355\\ 
			$0.4$ & 95 & 98 & 99 & 100 & 1.383 & $0.4$ & 80 & 82 & 88 & 94 & 2.302 \\
			$0.45$ & 96 & 98 & 98 & 100 & 1.503 & $0.45$ & 86 & 88 & 94 & 96 & 2.763 \\ 
			$0.5$ & 96 & 98 & 98 & 100 & 1.497 & $0.5$ & 86 & 88 & 94 & 96 & 2.791 \\ 
			$0.55$ & 100 & 100 & 100 & 100 & 2.557 & $0.55$ & 94 & 94 & 97 & 99 & 5.448 \\ 
			$0.6$ & 100 & 100 & 100 & 100 & 2.653 & $0.6$ & 95 & 95 & 97 & 99 & 5.685 \\ 
			$0.65$ & 100 & 100 & 100 & 100 & 3.790 & $0.65$ & 96 & 96 & 98 & 99 & 8.013 \\  
			$0.7$ & 100 & 100 & 100 & 100 & 5.996 & $0.7$ & 99 & 99 & 99 & 99 & 11.943 \\ \hline \hline
			\multicolumn{6}{|c||}{$n = 8, d = 4, s = 21$} & 	\multicolumn{6}{|c|}{$n = 9, d = 4, s = 19$} \\ \hline
			\multicolumn{1}{|c|}{\multirow{2}{*}{$M$}} & \multicolumn{4}{c|}{$tol$} & \multicolumn{1}{c||}{\multirow{2}{*}{time}} & \multicolumn{1}{|c|}{\multirow{2}{*}{$M$}} & \multicolumn{4}{c|}{$tol$} & \multicolumn{1}{c|}{\multirow{2}{*}{time}} \\ \cline{2-5} \cline{8-11}
			& 0.001 & 0.01 & 0.1 & 0.25  & & & 0.001 & 0.01 & 0.1 & 0.25  &  \\ \hline
			$c1$ &63 & 72 & 90 & 98 & 0.418 & $c1$ & 90 & 95 & 100 & 100 & 0.681 \\ 
			$c2$ & 97 & 99 & 100 & 100 & 4.275 & $c2$  & 98 & 100 & 100 & 100 & 2.687 \\ 
			$t1$ & 100 & 100 & 100 & 100 & 27.939 & $t1$ & 100 & 100 & 100 & 100 & 25.822 \\ 
			$0.4$ & 81 & 85 & 92 & 95 & 3.844 & $0.4$ & 92 & 93 & 97 & 98 & 2.490 \\ 
			$0.45$ & 86 & 88 & 98 & 98 & 4.737 & $0.45$ & 95 & 96 & 99 & 99 & 2.728 \\ 
			$0.5$ & 86 & 88 & 98 & 98 & 4.723 & $0.5$ & 95 & 96 & 99 & 99 & 2.734 \\ 
			$0.55$ & 96 & 96 & 98 & 99 & 10.024 & $0.55$ & 98 & 98 & 99 & 100 & 5.588 \\ 
			$0.6$ & 97 & 97 & 99 & 99 & 10.801 & $0.6$ & 98 & 98 & 99 & 100 & 5.800 \\ 
			$0.65$ & 98 & 99 & 100 & 100 & 14.416 & $0.65$ & 99 & 99 & 100 & 100 & 9.562 \\ 
			$0.7$ & 100 & 100 & 100 & 100 & 18.384 & $0.7$ & 100 & 100 & 100 & 100 & 14.853 \\ \hline \hline		
			\multicolumn{6}{|c||}{$n = 9, d = 4, s = 21$} & 	\multicolumn{6}{|c|}{$n = 9, d = 4, s = 23$} \\ \hline
			\multicolumn{1}{|c|}{\multirow{2}{*}{$M$}} & \multicolumn{4}{c|}{$tol$} & \multicolumn{1}{c||}{\multirow{2}{*}{time}} & \multicolumn{1}{|c|}{\multirow{2}{*}{$M$}} & \multicolumn{4}{c|}{$tol$} & \multicolumn{1}{c|}{\multirow{2}{*}{time}} \\ \cline{2-5} \cline{8-11}
			& 0.001 & 0.01 & 0.1 & 0.25  & & & 0.001 & 0.01 & 0.1 & 0.25  &  \\ \hline
			$c1$ & 73 & 79 & 93 & 96 & 0.700 & $c1$ & 60 & 72 & 95 & 98 & 0.815 \\ 
			$c2$ & 100 & 100 & 100 & 100 & 4.312 & $c2$  & 99 & 99 & 100 & 100 & 10.799 \\ 
			$t1$ & 100 & 100 & 100 & 100 & 48.051 & $t1$ & 100 & 100 & 100 & 100 & 87.657 \\ 
			$0.4$ & 87 & 89 & 95 & 95 & 3.518 & $0.4$ & 76 & 82 & 90 & 97 & 6.345 \\ 
			$0.45$ & 87 & 89 & 95 & 95 & 4.398 & $0.45$ & 80 & 87 & 94 & 99 & 8.318 \\ 
			$0.5$ & 87 & 89 & 95 & 95 & 4.416  & $0.5$ & 80 & 87 & 94 & 99 & 8.378 \\
			$0.55$ & 95 & 96 & 98 & 98 & 8.583 & $0.55$ & 92 & 96 & 98 & 99 & 21.360 \\ 
			$0.6$ & 96 & 97 & 99 & 99 & 9.081 & $0.6$ & 93 & 96 & 98 & 99 & 22.715 \\ 
			$0.65$ & 98 & 99 & 99 & 99 & 17.191 & $0.65$ & 97 & 99 & 100 & 100 & 35.411 \\ 
			$0.7$ & 100 & 100 & 100 & 100 & 26.776 & $0.7$ & 99 & 99 & 100 & 100 & 51.0111 \\ \hline  \hline
			\multicolumn{6}{|c||}{$n = 10, d = 4, s = 21$} & 	\multicolumn{6}{|c|}{$n = 10, d = 4, s = 23$} \\ \hline
			\multicolumn{1}{|c|}{\multirow{2}{*}{$M$}} & \multicolumn{4}{c|}{$tol$} & \multicolumn{1}{c||}{\multirow{2}{*}{time}} & \multicolumn{1}{|c|}{\multirow{2}{*}{$M$}} & \multicolumn{4}{c|}{$tol$} & \multicolumn{1}{c|}{\multirow{2}{*}{time}} \\ \cline{2-5} \cline{8-11}	
			& 0.001 & 0.01 & 0.1 & 0.25  & & & 0.001 & 0.01 & 0.1 & 0.25  &  \\ \hline
			$c1$ & 87 & 87 & 99 & 100 & 1.323& $c1$ & 83 & 89 & 99 & 99 & 1.402 \\ 
			$c2$ & 99 & 100 & 100 & 100 & 4.039 & $c2$ & 99 & 99 & 100 & 100 & 8.701 \\ 
			$t1$ & 100 & 100 & 100 & 100 & 60.871 & $t1$ & 100 & 100 & 100 & 100 & 119.759 \\ 
			$0.4$ & 86 & 88 & 96 & 98 & 4.345 & $0.4$ & 87 & 92 & 96 & 97 & 5.732 \\ 
			$0.45$ & 87 & 87 & 96 & 98 & 4.697 & $0.45$ & 89 & 92 & 96 & 97 & 7.699  \\ 
			$0.5$ & 87 & 87 & 96 & 98 & 4.720 & $0.5$ & 89 & 92 & 96 & 97 & 7.712 \\ 
			$0.55$ & 90 & 90 & 98 & 99 & 9.687 & $0.55$ & 96 & 98 & 100 & 100 & 16.947 \\ 
			$0.6$ & 91 & 91 & 98 & 99 & 10.301 & $0.6$ & 96 & 98 & 100 & 100 & 17.674 \\ 
			$0.65$ & 93 & 93 & 99 & 100 & 19.782 & $0.65$ & 98 & 99 & 100 & 100 & 37.969 \\ 
			$0.7$ & 100 & 100 & 100 & 100 & 32.182 & $0.7$ & 100 & 100 & 100 & 100 & 54.760 \\ \hline \hline
			\multicolumn{6}{|c||}{$n = 11, d = 4, s = 23$} & 	\multicolumn{6}{|c|}{$n = 11, d = 4, s = 25$} \\ \hline
			\multicolumn{1}{|c|}{\multirow{2}{*}{$M$}} & \multicolumn{4}{c|}{$tol$} & \multicolumn{1}{c||}{\multirow{2}{*}{time}} & \multicolumn{1}{|c|}{\multirow{2}{*}{$M$}} & \multicolumn{4}{c|}{$tol$} & \multicolumn{1}{c|}{\multirow{2}{*}{time}} \\ \cline{2-5} \cline{8-11}
			& 0.001 & 0.01 & 0.1 & 0.25  & & & 0.001 & 0.01 & 0.1 & 0.25  &  \\ \hline
			$c1$ & 84 & 96 & 99 & 99 & 2.454 & $c1$ & 68 & 82 & 95 & 97 & 2.582 \\ 
			$c2$ & 100 & 100 & 100 & 100 & 8.315 & $c2$ & 98 & 99 & 100 & 100 & 15.558\\ 
			$t1$ & 100 & 100 & 100 & 100 & 144.482 & $t1$ & 100 & 100 & 100 & 100 & 315.069 \\ 
			$0.4$ & 89 & 92 & 98 & 99 & 7.115 & $0.4$ & 76 & 80 & 90 & 96 & 10.223 \\ 
			$0.45$ & 90 & 91 & 97 & 99 & 8.482 & $0.45$ & 75 & 78 & 89 & 96 & 13.510 \\ 
			$0.5$ & 90 & 91 & 97 & 99 & 8.510 & $0.5$ & 75 & 78 & 89 & 96 & 13.533 \\ 
			$0.55$ & 96 & 96 & 99 & 100 & 16.579 & $0.55$ & 85 & 87 & 95 & 99 & 30.821 \\ 
			$0.6$ & 95 & 95 & 99 & 100 & 17.159 & $0.6$ & 85 & 87 & 95 & 99 & 32.885 \\ 
			$0.65$ & 98 & 98 & 98 & 100 & 41.036 & $0.65$ & 94 & 94 & 97 & 100 & 72.780 \\ 
			$0.7$ & 99 & 99 & 99 & 99 & 62.001 & $0.7$ & 96 & 96 & 98 & 100 & 114.329 \\ \hline	
		\end{tabular}
	\end{center} 
\end{table}

 \begin{table}[tbhp]
	\scriptsize
	\captionsetup{position=top} %<- Needed for using subtables created with the subfig package
	\caption{Numerical results on set IV: the amount of problems solved to the accuracy $tol$ by method $M$, i.e., with $\frac{|\Theta_{M} - \Theta_B|}{|\Theta_B|} \leq tol$ and the average computing time. For some polynomials from this set the refined TSSOS method with some parameter values $\tau$ generates an infeasible SDP relaxation. The amount of such problems is given in column $N_u$.}\label{tab:num3} 
	\begin{center}
		\begin{tabular}{||c||c|c|c|c|c|c||c|c|c|c|c|c||} \hline \hline
			\multicolumn{1}{||c||}{\multirow{3}{*}{$M$}} & \multicolumn{6}{|c||}{$(8,8,2,4,8,4)$} & 	\multicolumn{6}{|c||}{$(8,8,4,6,8,4)$} \\ \cline{2-13}
			& \multicolumn{4}{c|}{$tol$} & \multicolumn{1}{c|}{\multirow{2}{*}{time}}&\multicolumn{1}{c||}{\multirow{2}{*}{$N_u$}} & \multicolumn{4}{c|}{$tol$} & \multicolumn{1}{c|}{\multirow{2}{*}{time}}&\multicolumn{1}{c||}{\multirow{2}{*}{$N_u$}} \\ \cline{2-5} \cline{8-11}
			& 0.001 & 0.01 & 0.1 & 0.25  & & & 0.001 & 0.01 & 0.1 & 0.25  &  &\\ \hline
			$c1$ & 14 & 20 & 40 & 59 & 3.96 & 0 & 20 & 31 & 56 & 76 & 6.40 & 0\\ 
			$c2$ & 51 & 57 & 71 & 78 & 37.45 & 0 &  65 & 75 & 93 & 98 & 350.84 & 0\\ 
			$t1$ & 90 & 92 & 98 & 100 & 91.90 & 0 & 88 & 90 & 94 & 97 & 309.98 & 0\\ 
			$0.4$ & 27 & 36 & 63 & 74 & 5.34 & 5 & 47 & 55 & 73 & 84 & 8.89 & 3 \\ 
			$0.45$ & 29 & 41 & 65 & 77 & 5.69 & 5 & 54 & 61 & 77 & 84 & 11.26 & 3 \\ 
			$0.5$ & 29 & 42 & 65 & 77 & 6.11 & 5 & 56 & 62 & 77 & 87 & 12.72 & 3 \\ 
			$0.55$ & 45 & 53 & 75 & 81 & 10.24 & 3 & 68 & 75 & 87 & 93 & 21.86 & 1 \\ 
			$0.6$ & 45 & 53 & 75 & 81 & 10.23 & 3 & 70 & 76 & 88 & 92 & 30.60 & 1 \\ 
			$0.65$ & 53 & 61 & 81 & 88 & 16.66 & 2 &  73 & 77 & 89 & 94 & 48.66 & 1 \\
			$0.7$ & 61 & 69 & 88 & 95 & 28.35 & 0 & 77 & 80 & 89 & 95 & 71.72 & 1 \\ 
			$0.75$ & 68 & 73 & 91 & 96 & 35.53 & 0 & 77 & 82 & 91 & 96 & 99.77 & 1 \\ \hline \hline
		\end{tabular}
	\end{center} 
\end{table}

\begin{table}[tbhp]
	\scriptsize
	\captionsetup{position=top} %<- Needed for using subtables created with the subfig package
	\caption{Numerical results for the chordal ($c1$ and $c2$), block ($t1$) and refined TSSOS with different parameter values $\tau_0 = \tau_1 = \tau$, $\tau \in \set{0.1, 0.2, 0.3, 0.4, 0.5, 0.6, 0.7}$ in the constrained case with $\mathbf{K} := \{(x_1, \ldots, x_n) \in \mathbb{R}^n \; | \; g_1 = 25 - (x_1^2 + \cdots + x_n^2) \geq 0\}$,  the amount of problems not solved within the time limit is given in column $N_u$.} \label{tab:num4} 
	\begin{center}
		\begin{tabular}{||c||c|c|c|c|c|c||c|c|c|c|c|c||} \hline \hline
			\multicolumn{1}{||c||}{\multirow{3}{*}{$M$}} & \multicolumn{6}{|c||}{$n = 8, d = 4, s = 21, \hat{d} = 4$} & 	\multicolumn{6}{|c||}{$n = 9, d = 4, s = 21, \hat{d} = 4$} \\ \cline{2-13}
			& \multicolumn{4}{c|}{$tol$} & \multicolumn{1}{c|}{\multirow{2}{*}{time}}&\multicolumn{1}{c||}{\multirow{2}{*}{$N_u$}} & \multicolumn{4}{c|}{$tol$} & \multicolumn{1}{c|}{\multirow{2}{*}{time}}&\multicolumn{1}{c||}{\multirow{2}{*}{$N_u$}} \\ \cline{2-5} \cline{8-11}
			& 0.001 & 0.01 & 0.1 & 0.25  & & & 0.001 & 0.01 & 0.1 & 0.25  & & \\ \hline
			$c1$ & 53 & 61 & 82 & 90 & 2.13 & 0 & 64 & 71 & 90 & 98 & 2.35 & 0\\ 
			$c2$ & 98 & 99 & 100 & 100 & 122.79 & 0 & 99 & 99 & 100 & 100 & 287.51 & 0 \\ 
			$t1$ & 100 & 100 & 100 & 100 & 68.61 & 0 &  100 & 100 & 100 & 100 & 175.58 & 0\\ 
			$0.1$ & 44 & 46 & 65 & 80 & 2.61 & 0 &  73 & 76 & 88 & 92 & 3.28 & 0\\ 
			$0.2$ & 58 & 60 & 75 & 87 & 2.83 & 0 & 73 & 75 & 85 & 91 & 3.35 & 0\\ 
			$0.3$ & 67 & 70 & 84 & 89 & 3.77 & 0 & 78 & 80 & 89 & 94 & 3.90 & 0\\ 
			$0.4$ & 84 & 86 & 91 & 97 & 6.08 & 0 & 81 & 82 & 90 & 95 & 5.76 & 0\\ 
			$0.5$ & 92 & 95 & 95 & 99 & 10.80 & 0 & 87 & 90 & 94 & 96 & 8.00 & 0\\ 
			$0.6$ & 98 & 99 & 100 & 100 & 30.33 & 0 & 93 & 94 & 97 & 98 & 17.18 & 0\\
			$0.7$ & 100 & 100 & 100 & 100 & 45.04 & 0 & 99 & 99 & 100 & 100 & 34.36 & 0 \\\hline \hline
			\multicolumn{1}{||c||}{\multirow{3}{*}{$M$}} & \multicolumn{6}{|c||}{$n = 9, d = 4, s = 23, \hat{d} = 4$} & 	\multicolumn{6}{|c||}{$n = 10, d = 4, s = 23, \hat{d} = 4$} \\ \cline{2-13}
			& \multicolumn{4}{c|}{$tol$} & \multicolumn{1}{c|}{\multirow{2}{*}{time}}&\multicolumn{1}{c||}{\multirow{2}{*}{$N_u$}} & \multicolumn{4}{c|}{$tol$} & \multicolumn{1}{c|}{\multirow{2}{*}{time}}&\multicolumn{1}{c||}{\multirow{2}{*}{$N_u$}} \\ \cline{2-5} \cline{8-11}
			& 0.001 & 0.01 & 0.1 & 0.25  & & & 0.001 & 0.01 & 0.1 & 0.25  & & \\ \hline
			$c1$ & 35 & 38 & 63 & 78 & 3.22 & 0 & 51 & 58 & 80 & 92 & 3.73 & 0\\ 
			$c2$ & 97 & 99 & 100 & 100 & 914.31 & 0 & 96 & 96 & 97 & 97 & 862.79 & 2 \\ 
			$t1$ & 100 & 100 & 100 & 100 & 258.63 & 0 &  100 & 100 & 100 & 100 & 432.98 & 0\\ 
			$0.1$ & 48 & 49 & 67 & 77 & 4.17 & 0 & 61 & 64 & 81 & 87 & 4.96 & 0\\ 
			$0.2$ & 50 & 51 & 71 & 84 & 4.48 & 0 & 63 & 64 & 81 & 88 & 5.30 & 0\\ 
			$0.3$ & 57 & 58 & 76 & 84 & 6.06 & 0 & 66 & 69 & 82 & 89 & 6.10 & 0\\ 
			$0.4$ & 69 & 75 & 86 & 92 & 8.93 & 0 & 70 & 72 & 85 & 91 & 7.91 & 0\\ 
			$0.5$ & 84 & 87 & 94 & 97 & 16.20 & 0 & 80 & 82 & 93 & 95 & 11.10 & 0\\ 
			$0.6$ & 94 & 97 & 100 & 100 & 45.66 & 0 & 87 & 87 & 96 & 96 & 31.40 & 0\\
			$0.7$ & 98 & 100 & 100 & 100 & 96.83 & 0 & 96 & 97 & 99 & 99 & 58.31 & 0 \\\hline \hline
		\end{tabular}
	\end{center} 
\end{table}

\begin{figure}[tbhp]
	\centering 
	\includegraphics[trim = 110 50 140 100, clip, width=1\linewidth]{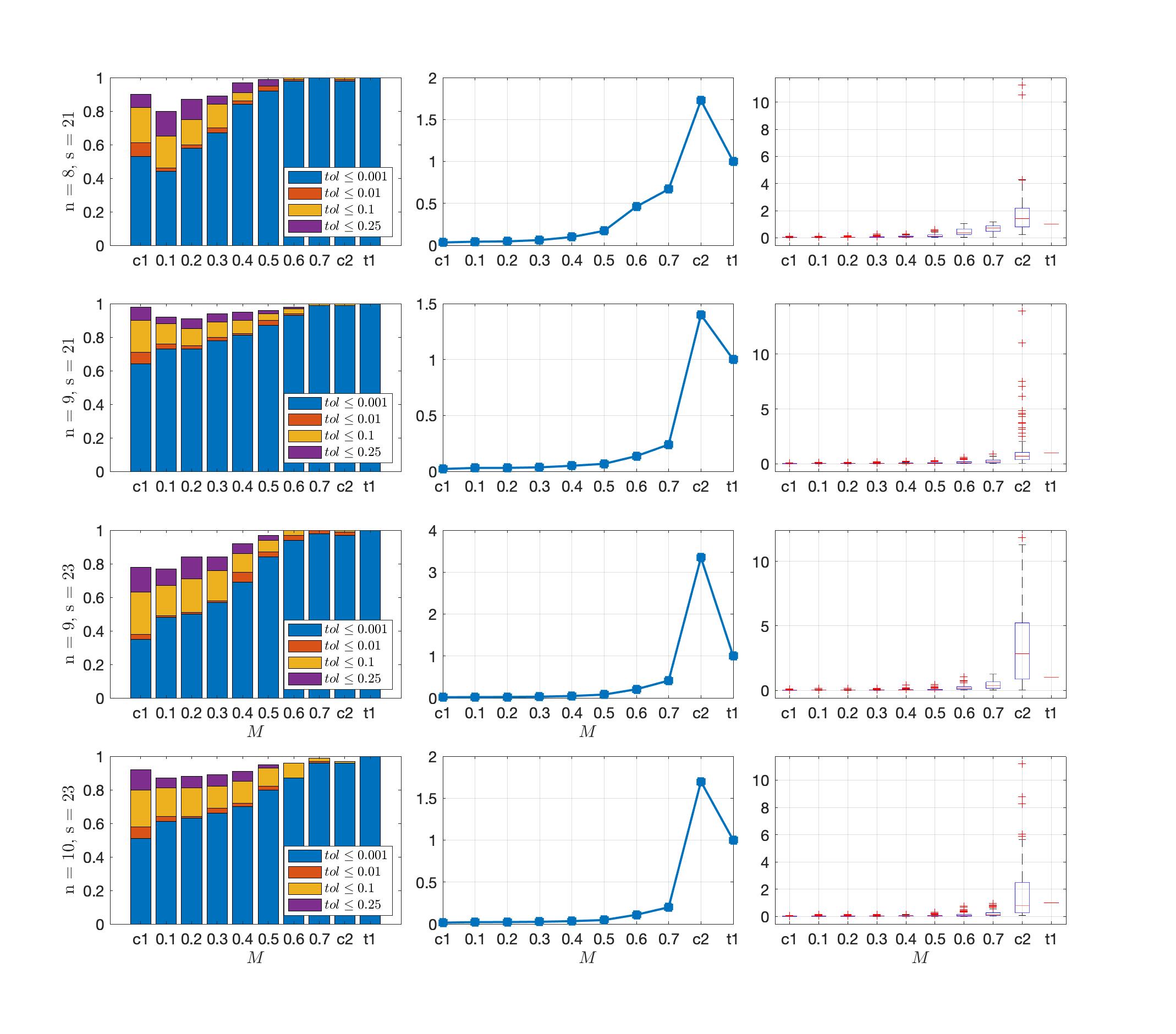}
	\vspace{-30pt}
	\caption{Numerical results on polynomials from set I for the chordal ($c1$ and $c2$), block ($t1$) and refined TSSOS with different parameter values $\tau_0 = \tau_1 = \tau$, $\tau \in \set{0.1, 0.2, 0.3, 0.4, 0.5, 0.6, 0.7}$ with $\mathbf{K} := \{(x_1, \ldots, x_n) \in \mathbb{R}^n \; | \; g_1 = 25 - (x_1^2 + \cdots + x_n^2) \geq 0\}$ the relaxation order $\hat{d} = 4$: plots in column 1 show the fraction of 100 problems solved to a certain accuracy by method $M$, i.e., with $\frac{|\Theta_{M} - \Theta_B|}{|\Theta_B|} \leq tol$, the ratio of CPU time for method $M$ to CPU time of TSSOS ($k=1$), i.e., $T_{M} / T_{t1}$, averaged over 100 problems is depicted in column 2, box plots showing the spread of values $T_{M} / T_{t1}$ are given in column 3, where the central mark indicates the median, and the bottom and top edges of the box indicate the 25th and 75th percentiles, respectively. Note that the CPU time limit for the SDP solver is 5000 seconds. If for some problem the solver fails to terminate within this time on a relaxation produced by method $M$, this problem is considered to be unsolved by this method. }
	\label{fig:con}
\end{figure}

\bibliographystyle{siamplain}
\bibliography{references}

\begin{thebibliography}{10}

\bibitem{Ahmadi}
{\sc A.~A. Ahmadi and A.~Majumdar}, {\em {DSOS} and {SDSOS} optimization: {LP}
  and {SCOCP}-based alternatives to sum of squares optimization}, 48th Annual
  Conference on Information Sciences and Systems ({CISS}),  (2014).

\bibitem{Choi}
{\sc M.-D. Choi, T.~Lam, and B.~Reznick}, {\em {Sums of Squares of Real
  Polynomials}}, {Proceedings of Symposia in Pure mathematics, AMS}, 58 (1995),
  pp.~103--126.

\bibitem{JuMP}
{\sc I.~Dunning, J.~Huchette, and M.~Lubin}, {\em {JuMP: A Modeling Language
  for Mathematical Optimization}}, {SIAM Review}, 59.

\bibitem{Krivine}
{\sc J.-L. Krivine}, {\em Anneaux préordonnés}, J. Anal. Math., 12 (1964),
  pp.~307--326.

\bibitem{Lasserre1}
{\sc J.-B. Lasserre}, {\em Global optimization with polynomials and the problem
  of moments}, SIAM Journal on Optimization, 11 (2001), pp.~796--817.

\bibitem{CorSparsity1}
{\sc J.-B. Lasserre}, {\em Convergent {SDP}-relaxations in polynomial
  optimization with sparsity}, SIAM Journal on Optimization, 3 (2006).

\bibitem{Lasserre2}
{\sc J.-B. Lasserre}, {\em {An Introduction to Polynomial and Semi-Algbraic
  Optimization}}, Cambridge University Press, Cambridge, UK, 2015.

\bibitem{Lasserre3}
{\sc J.-B. Lasserre, K.-C. Toh, and S.~Yang}, {\em A bounded degree {SOS}
  hierarchy for polynomial optimization}, EURO J. Comput. Optim., 5 (2017),
  pp.~87--117.

\bibitem{Mosek}
{\sc {MOSEK ApS}}, {\em {MOSEK Optimization Toolbox}}, 2020,
  \url{https://docs.mosek.com/8.1/toolbox/index.html#}.

\bibitem{Nie}
{\sc J.~Nie}, {\em Optimality conditions and finite convergence of {L}asserre's
  hierarchy}, Math. Program., 146 (2014), pp.~97--121.

\bibitem{Putinar}
{\sc M.~Putinar}, {\em {Positive Polynomials on Compact Semi-algebraic Sets}},
  Indiana University Mathematics Journal, 42 (1993), pp.~969--984.

\bibitem{Reznick}
{\sc B.~Reznick}, {\em Extremal {PSD} forms with few terms}, Duke Math. J., 45
  (1978), pp.~363--374.

\bibitem{Stengle}
{\sc G.~Stengle}, {\em {A Nullstellensatz and a Positivstellensatz in
  semialgebraic geometry}}, Math. Ann., 207 (1974), pp.~87--97.

\bibitem{CorSparsity2}
{\sc H.~Waki, S.~Kim, M.~Kojima, and M.~Muramatsu}, {\em Sums of squares and
  semidefinite program relaxations for polynomial optimization problems with
  structured sparsity}, SIAM Journal on Optimization, 1 (2016), pp.~218--242.

\bibitem{TSSOStool}
{\sc J.~Wang}, {\em {TSSOS} tool}, \url{https://github.com/wangjie212/TSSOS}.

\bibitem{TSSOS2}
{\sc J.~Wang, H.~Li, and B.~Xia}, {\em A new sparse {SOS} decomposition
  algorithm based on term sparsity}, Proceedings of the 2019 on International
  Symposium on Symbolic and Algebraic Computation,  (2019).

\bibitem{ChordalTSSOS}
{\sc J.~Wang, V.~Magron, and J.-B. Lasserre}, {\em Chordal-{TSSOS}: a
  moment-{SOS} hierarchy that exploits term sparsity with chordal extension},
  SIAM Journal on Optimization, 31 (2021).

\bibitem{TSSOS}
{\sc J.~Wang, V.~Magron, and J.-B. Lasserre}, {\em {TSSOS: A Moment-SOS
  Hierarchy That Exploits Term Sparsity}}, SIAM Journal on Optimization, 31
  (2021), \url{https://doi.org/10.1137/19M1307871}.

\end{thebibliography}
\end{document}